\documentclass[bibay3]{asl}

\title[Elementary recursive algorithms]{Elementary recursive algorithms${\,}^\ast$}
\author{Yiannis N. Moschovakis}
\twoaddress{
Department of Mathematics\\
University of California,
Los Angeles, CA, USA}
{Department of Mathematics,
University of Athens, Greece}{3}

\email{ynm@math.ucla.edu}
\urladdr{http://www.math.ucla.edu/\urltilde ynm}

\makeatletter
\newcount\currhour
\newcount\currminutes
\newcount\temp
\currhour=\time\currminutes=\time
\divide\currhour by 60 
\temp=\currhour
\multiply\temp by 60    
\advance\currminutes by -\temp
\def\minutes{\ifnum\currminutes<10 0\number\currminutes%
\else\number\currminutes\fi}
\def\now{\today, \number\currhour:\minutes}

\def\bottomnote{\@ifnextchar
    [{\@xfootnotenext}{\xdef\@thefnmark{}\@footnotetext}}
\def\@xfootnotenext[#1]{\begingroup \csname c@\@mpfn\endcsname #1\relax
   \xdef\@thefnmark{\thempfn}\endgroup \@footnotetext}

\newcommand\eg{e.g., }
\newcommand\ie{i.e., }
\newcommand\cf{cf.\ }
\newcommand\ulp{\textup{(}}
\newcommand\urp{\textup{)}}
\newcommand\set{\textup{set}}
\newcommand\recto{\rightsquigarrow}
\newcommand\recursor{\mathfrak{r}}
\newcommand\refint{\textup{int}}
\newcommand{\where}{{\sf~where~}}
\newcommand\ycond{\textsf{cond}}
\newcommand\redto{\Rightarrow}
\newcommand\str[1]{\textup{\textbf{#1}}}
\newcommand\A{{\str{A}}}
\newcommand\B{{\str{B}}}

\newcommand\f[1]{\textup{\textsf{#1}}}
\newcommand{\inj}{\rightarrowtail}
\newcommand{\surj}{\mbox{\kern .2em$\rightarrow$\hspace*{-.8em}$\rightarrow$\kern .2em}}
\newcommand{\bij}{\mbox{\,$\rightarrowtail\kern -.8em \rightarrow$\,}}
\newcommand{\pfto}{\rightharpoonup}
\newcommand\bool{{\normalfont\texttt{boole}}}
\newcommand\ind{{\normalfont\texttt{ind}}}
\newcommand\boolset{\mathbb{B}}
\newcommand\booleans{{\{\true,\false\}}}
\def\true{{\textup{t\hskip -1.7pt t}}}
\def\false{{\textup{ff}}}
\newcommand\arity{\textup{arity}}

\newcommand\sort{{\textup{sort}}}
\newcommand\Lbrace{\Big\{}
\newcommand\Rbrace{\Big\}}
\newcommand\den{\textup{den}}
\newcommand\LS{{\str{L}^\ast}}
\newcommand\nil{\textup{nil}}
\newcommand\tail{\textup{tail}}
\newcommand\eq{{\textup{eq}}}
\newcommand\head{\textup{head}}
\newcommand\cons{\textup{cons}}
\def\tif{\textup{if }}
\def\tthen{\textup{then }}
\def\telse{\textup{else }}

\newcommand\nats{\mathbb{N}}

\newcommand\rationals{\mathbb{Q}}
\newcommand\reals{\mathbb{R}}

\newcommand{\converges}{\kern -.2em\downarrow\kern .1em}
\newcommand{\diverges}{\kern-.1em\uparrow\kern .05em}

\def\st{\,\mid\,}
\newcommand{\edf}{=_{\mathrm{df}}}
\newcommand\conj{{~\&~}}
\newcommand{\emptyproduct}{{\bf I}}
\newcommand{\oo}{\overline}

\newlength\wherelength
\newcommand\twhere{\settowidth{\wherelength}{$\textup{where}$}
\,\textup{where}\hskip-\wherelength%
\raisebox{1.3ex}{\rule{\wherelength}{.3pt}}\,}
\newcommand{\rwhere}{\textit{ where }}
\newcommand\nf{\textup{cf}}
\newcommand\iterator{\mathfrak{i}}

\newcommand\iterto{\rightsquigarrow}

\newcommand\hint[2][]{\textup{\textsc{Hint: }{#2}}}
\long\def\sol#1{}
\newcommand\size{{\textup{size}}}
\newcommand\yinput{\textup{input}}
\newcommand\youtput{\textup{output}}
\newcommand\etc{etc. }

\newcommand\inteq[1]{=^{#1}_{\textup{int}}}
\newcommand{\eedf}{\iff\kern -.65em_{\mathrm{df}}~}
\newcommand\ylength{\textup{length}}
\newcommand\id{\textup{id}}
\newcommand\vinj{{\textup{inj}}}
\newcommand\EQP{{\mathcal{E}}}

\newcommand\wwedge{\mbox{\rlap{$\bigwedge$}\hskip 3pt$\bigwedge$\,}}
\newcommand\prog[1]{\widetilde{#1}}
\makeatother

\makeatletter
\def\plaindraft#1{%
\gdef\draftinfo{#1}
\let\oldps@headings\ps@headings%
\def\ps@headings{\oldps@headings%
\def\@oddfoot{\vbox to 0pt{\vskip 10pt\noindent \small\tt
#1 }}%
\def\@evenfoot{\vbox to 0pt{\vskip 10pt\noindent \small\tt
#1 }}}%
\AtBeginDocument{\ttdraft}%
\pagestyle{headings}}

\def\myplaindraft#1{%
\gdef\draftinfo{#1}
\let\oldps@myheadings\ps@myheadings%
\def\ps@myheadings{\oldps@myheadings%
\def\@oddfoot{\vbox to 0pt{\vskip 10pt\noindent \small\tt
#1 }}%
\def\@evenfoot{\vbox to 0pt{\vskip 10pt\noindent \small\tt
#1 }}}
\pagestyle{myheadings}}

\def\ttdraft{%
\let\oldps@firstpage\ps@firstpage%
\gdef\ps@firstpage{\oldps@firstpage
\def\@oddfoot{\vbox to 0pt{\vskip 10pt\noindent
\small\tt \draftinfo}}}}
\makeatother
\newcommand\problems{\protect\vspace*{10pt}\ysection{Problems for Section~\thesection}}

\numberwithin{equation}{section}
\renewcommand\theequation{\thesection-\arabic{equation}}

\def\figurename{Diagram}

\theoremstyle{definition}
\newtheorem{prob}{\addtocounter{dprob}{1}\addtocounter{openprob}{1}\bfseries Problem}[section]

\newtheorem{dprob}{\addtocounter{prob}{1}\addtocounter{openprob}{1}\bfseries Problem}[section]

\newtheorem{note}{}[section]

\theoremstyle{plain}
\newtheorem{theorem}[note]{\bfseries Theorem}
\newtheorem{proposition}[note]{\bfseries Proposition}
\newtheorem{lemma}[note]{\bfseries Lemma}
\newtheorem{corollary}[note]{\bfseries Corollary}

\plaindraft{Elementary recursive algorithms, draft, \hfill
\now\quad p.~\thepage}

\newcommand\ysection[1]{\subsubsection*{\textbf{#1}}}
\newcommand\ar{ARIC}
\newcommand\tsigma{\ddot{\sigma}}
\newcommand\tu[1]{\textup{(#1)}}
\newcommand\tPhi{\Phi^2}

\newcommand\fx{{\f w}}
\newcommand\fs{{\f u}}
\newcommand\ft{{\f v}}
\newcommand\ef[2]{{\textup{[#1-#2]}}}

\begin{document}
\mbox{}
\vspace*{-20pt}

\maketitle

\vspace*{-15pt}
\tableofcontents
\vspace*{-20pt}

\bottomnote{${}^\ast$A preliminary version of the results in this
article was included in an early draft of \cite{aric2019} (ARIC)
as an additional Chapter in Part 1. It was replaced by a brief
summary in Section 2H of the final, published version of \ar,
because it is not directly relevant to the main aim of that
monograph, which is to develop methods for deriving and justifying
robust lower complexity bounds for mathematical problems. There
are many references in the sequel to \ar\ and to older work by
several people, but I have included enough of the  basic
definitions and facts so that the statements and proofs of the new
results in Sections~\ref{cforms} and~\ref{decidability} stand on
their own.}

They run our lives, if you believe the hype in the news, but there is no
\looseness=-1
precise definition of \textit{algorithms} which is generally
accepted by the mathematicians, logicians and computer
scientists who create and study them.\footnote{Using imprecise formulations of the
\textit{Church-Turing Thesis} and vague references
to~\cite{church1935,church1936} and~\cite{turing1936}, it is
sometimes claimed naively that \textit{algorithms are Turing
machines}. This does not accord with the original formulations of
the Church-Turing Thesis, \cf the discussion in Section~1.1
of~\cite{ynm2014} (which repeats points well-known and understood
by those who have thought about this matter); and it is not a
useful assumption for algebraic complexity theory, \cf page~2 of
\ar.}
My main aims here are (first) to discuss briefly and point to the
few publications that try to deal with this foundational question
and (primarily) to outline in Sections~\ref{cforms}
and~\ref{decidability} simple proofs of two basic mathematical
results about the \textit{elementary recursive algorithms from
specified primitives} expressed by \textit{recursive} (McCarthy)
\textit{programs}.

\section{What is an algorithm?}
\label{algorithmssec}

With the (standard, self-evident) notation of Sections~1D and~\ref{eqlogic} of \ar, we
will focus on algorithms which compute partial functions and
(decide partial) relations
\begin{equation}
\label{sortedpfs}
f : A^n\pfto A_s \qquad(s\in\{\ind,\bool\}, A_\ind=A, A_\bool=\booleans)
\end{equation}
from the finitely many primitives of a (partial, typically infinite) $\Phi$-structure
\begin{equation}
\label{p-structure}
\A = (A, \{\phi^\A\}_{\phi\in\Phi})\quad(\phi\in\Phi, \phi^\A : A^{\arity(\phi)}\pfto A_{\sort(\phi)}).
\end{equation}

The most substantial part of this restriction is that it leaves
out algorithms with side effects and interaction, \cf the footnote
on page~\pageref{elementary} of \ar\ and the relevant Section 3B
in~\cite{ynmflr}.

Equally important is the restriction to \textit{algorithms from
specified primitives}, especially as the usual formulations of the
\textit{Church-Turing Thesis} suggest that the primitives of a
Turing machine are in some sense ``absolutely computable'' and
need not be explicitly assumed or justified. We have noted in the
Introduction to ARIC (and in several other places) why this is not
a useful approach; but in trying to understand
\textit{computability} and the meaning of the Church-Turing
Thesis, it is natural to ask whether there are absolutely
computable primitives and what those might be. See Sections 2 and
8 of~\cite{ynm2014} for a discussion of the problem and references
to relevant work, especially the eloquent analyses
in~\cite{gandy1980} and~\cite{kripke2000}.

There is also the restriction to \textit{first-order primitives},
partial functions and relations. This is necessary for the
development of a conventional theory of complexity, but recursion
and computability from higher type primitives have been
extensively studied: see~\cite{kleene1959},
\cite{ynmkechris-moschovakis} and~\cite{sacks1990} for the
higher-type recursion which extends directly the first-order
notion we have adopted in ARIC, and \cite{longley-normann2015} for a
near-complete exposition of the many and different approaches to
the topic.\footnote{See also~\cite{ynmflr}---which is about
recursion on structures with arbitrary monotone functionals for
primitives---and the subsequent~\cite{ynmbotik} where the
relevant notion of \textit{algorithm from higher-type
primitives} is modeled rigorously.}

Once we focus on algorithms which compute partial functions and
relations as in~\eqref{sortedpfs} from the primitives of a
$\Phi$-structure, the problem of modeling them by
set-theoretic  objects comes down basically to choosing between
\textit{iterative algorithms} specified by (classical) computation models as
in Section~\ref{cmodels} of \ar\ and \textit{elementary recursive
algorithms} expressed directly by recursive (McCarthy) programs; at least
this is my view, which I have explained and defended as best I can
in Section~3 of~\cite{ynmsicily}.

The first of these choices---that algorithms are iterative
processes---is \textit{the standard view},
explicitly or implicitly adopted  (sometimes
with additional restrictions) by most mathematicians and computer
scientists, including Knuth in Section 1.1 of his
classic~\cite{knuth1973}. More recently (and substantially more
abstractly, on arbitrary structures), this standard view has been
developed, advocated and defended by Gurevich and his
collaborators, \cf~\cite{gurevich1995,gurevich2000} and
\cite{gurevich2008}; see also \cite{zt2000} and \cite{duzi2014}.

I have made the second choice---that algorithms are directly
expressed by systems of mutual recursive definitions---and I have developed
and defended this view in several papers,
including~\cite{ynmsicily}. I will not repeat these arguments
here, except for the few remarks in the remainder of this Section
about the role that iterative algorithms play in the theory of
recursion and (especially) Proposition~\ref{recit}, which verifies
that iterative algorithms are ``faithfully coded'' by the
recursive algorithms they define, and so their theory is
not ``lost'' when we take recursive algorithms to be the basic
objects.

By the definitions in  Section~\ref{synsem} of \ar\ (reviewed in
Section~\ref{elptcond} below), a \textbf{recursive} (McCarthy)
$\Phi$-\textbf{program} is a syntactic expression
\begin{equation}
\label{recprogfirst}
E\equiv E_0\where\Lbrace\f p_1(\vec{\f
x}_1)= E_1, \ldots, \f p_K(\vec{\f x}_K)=E_K\Rbrace,
\end{equation}
where $E_0, E_1, \ldots, E_k$ are (pure, explicit) $\Phi\cup\{\f
p_1,\ldots,\f p_K\}$-terms and for every $i=1, \ldots, K$ all the individual
variables which occur in $E_i$ are included in the list $\vec{\f x}_i$ of distinct
individual variables; and an \textbf{extended program} is a pair
\begin{equation}
\label{recprogext}
(E,\vec{\f x}) \equiv E(\vec{\f x})\equiv E_0(\vec{\f x}) \where\Lbrace\f p_1(\vec{\f
x}_1)= E_1, \ldots, \f p_K(\vec{\f x}_K)=E_K\Rbrace
\end{equation}
of a program $E$ and a list of distinct individual variables $\vec{\f x}$ which includes
all the individual variables which occur in $E_0$.

\ysection{Recursive algorithms}

My understanding of \textit{the algorithm defined by $E(\vec{\f
x})$ in a $\Phi$-structure} $\A$ is that it calls for solving in
$\A$ the system of recursive equations in the \textit{body} of
$E(\vec{\f x})$ (within the braces $\{\cdots\}$) and then plugging
the solutions in its \textit{head} $E_0(\vec{\f x})$ to compute
$\den(\A,E(\vec x))$, the value of
the partial function computed by $E(\vec{\f x})$ on the input $\vec x$;
\textit{how} we find the canonical (least) solutions of this
system is not specified in this view by the algorithm from the
primitives of $\A$ defined by $E(\vec{\f x})$.

This ``lack of specificity'' of elementary recursive algorithms is
surely a weakness of the view I have adopted, especially as it
leads to some difficult problems.

\ysection{Implementations}

To compute $\den(\A,E(\vec x))$ for $\vec x\in A^n$, we might use
the method outlined in the proof of the Fixed Point
Lemma~\ref{lfp} or the recursive machine defined in
Section~\ref{cmodels} of \ar, or any one of several well known and
much studied \textit{implementations of recursion}. These are
iterative algorithms defined on structures which are (typically)
richer than $\A$ and they must satisfy additional properties
relating them to $E(\vec{\f x})$---we would not call any iterative
algorithm from any primitives which happens to compute the same
partial function on $A$ as $E(\vec{\f x})$ an implementation of
it; so to specify \textit{the implementations of a recursive
program} is an important (and difficult) part of this approach to
the foundations of the theory of algorithms, \cf~\cite{ynmsicily}
and (especially) \cite{ynmpaschalis2008} which proposes a precise
definition and establishes some basic results about it.

On the other hand, the standard view has some problems of its own:

\ysection{Recursion vs.\ iteration}

Iterators are defined rigorously in \ar\ and
Theorem~\ref{itrecthmA} gives a strong, precise version of the
slogan
\begin{equation}
\label{itisrec}
\textit{iteration can be reduced to recursion};
\end{equation}
\textit{on every structure $\A$, if $f:A^n\pfto A_s$ is computed
by an $\A$-explicit iterator, then $f$ is also computed by an
$\A$-recursive program}. The converse of~\eqref{itisrec} is not
true however: there are structures where some $\A$-recursive
relation is not \textit{tail recursive}, \ie it cannot be decided
by an iterator which is explicit in $\A$---it is necessary to add
primitives and/or to enlarge the universe of $\A$. Classical
examples include \cite{patterson-hewitt1970}
(Theorem~\ref{pathewhm} in \ar), \cite{lynch-blum1979} and (the
most interesting) \cite{tiuryn1989}.\footnote{See also
\cite{jones1999,jones2001} and
\cite{bhaskar2017int,bhaskar2017ext}.} The last
Chapter~\ref{algebra} of \ar\ discusses a (basic) problem from
\textit{algebraic complexity theory}, a very rich and active
research area in which the recursive representation of algorithms
is essential.

\ysection{Functional vs.\ imperative programming}

It is sometimes argued that the main difference between recursion
and iteration is a matter of ``programming style'', at least for
the structures which are mostly studied in complexity theory in
which every recursive partial function is \textit{tail recursive},
\ie computed by some $\A$-explicit iterator. This is not quite
true: consider, for example the classical \textit{merge-sort}
algorithm (Section~\ref{algexamples} of \ar) which sorts
strings from the primitives of the Lisp structure
\[
\LS = (L^\ast,\nil,\eq_{\nil},\head,\tail,\cons)
\]
over an ordered set $L$, \eqref{stringstr} in \ar; it is certainly
implemented by some $\LS$-explicit  iterator (as is every
recursive algorithm of $\LS$), but which one of these \textit{is}
the merge-sort? It seems that we can understand this classic
algorithm and reason about it better by looking at
Proposition~\ref{mergelemma} of \ar\ rather than focussing on
choosing some specific implementation of it.\footnote{See also
Theorem~\ref{sortinglb} of \ar\ for a precise formulation and
proof of the basic optimality property of the merge-sort.}

\ysection{Proofs of correctness}

In Proposition~\ref{mergesortprop} of \ar, we claimed the correctness of
the merge-sort---that it sorts---by just saying that

\begin{quotation}
\it The sort function satisfies the equation \ldots
\end{quotation}
whose proof was too trivial to deserve more than
the single sentence
\begin{quotation}
\it The validity of~\eqref{mergesort} is immediate, by induction on $|u|$.
\end{quotation}
To prove the correctness of an iterator that ``expresses the
merge-sort'', you must first design a specific one
and then explain how you can extract from all the
``housekeeping'' details necessary for the specification of
iterators a proof that what is actually being computed is the
sorting function; most likely you will trust that a formal version
of~\eqref{mergesort} of \ar\ is implemented correctly by some compiler or
interpreter of whatever higher-level language you are
using, as we did for the recursive machine.\smallskip

Simply put, whether correct or not, the view that algorithms are
faithfully expressed by systems of recursive equations typically
separates proofs of \textit{their correctness} and their basic
\textit{complexity properties} which involve only purely
mathematical facts from the relevant subject and standard results
of fixed-point-theory, from proofs of \textit{implementation
correctness} for programming languages which are ultimately
necessary but have a very different flavor.

\ysection{Defining mathematical objects in set theory}

Finally, it may be useful to review here briefly what it means to
\textit{define algorithms} according to a general (and widely
accepted, I think) naive view of what it means to \textit{define
mathematical objects}. This is discussed more fully in Section~3
of~\cite{ynmsicily}.\smallskip

One standard example is the ``definition'' (or ``construction'')
of \textit{real numbers} using (canonically) convergent sequences
of rational numbers: we set first
\begin{multline*}
x\in\textup{Cauchy}(\rationals) \iff x:\nats\to\rationals
\\
\conj (\forall k)(\exists n)(\forall i,j)
\Big[
i,j\geq n \implies |x(i)- x(j)|<\frac{1}{k+1}\Big],
\end{multline*}
next we put for $x,y\in\textup{Cauchy}(\rationals)$
\[
x\sim y \iff (\forall k)(\exists n)(\forall i>n)|x(i)-y(i)|<
\frac{1}{k+1},
\]
and finally we declare that $x,y\in\textup{Cauchy}(\rationals)$
\textit{represent} the same real number if $x\sim y$.

In general, a \textit{representation in set theory} of a
mathematical notion $P$ is a pair of a set (or class) $\mathcal
C_P$ of set-theoretic objects which represent the objects that
fall under $P$ and an equivalence condition $\sim_P$ on $\mathcal
C_P$ which models the identity relation on them; and the
representation is \textit{faithful}---and can be viewed as a
\textit{definition of $P$ in set theory}---to the extent that the
relations on $P$-objects that we want to study are those
which are equivalent to the $\sim_P$-invariant properties of the
objects in $\mathcal C_P$.

Now, number theorists could not care less about such
``definitions'' of real numbers and they were happily
investigating the existence and properties of rational, algebraic
and transcedental numbers for more than two thousand years before
they were given in the 1870s. Part of the reason for giving them
was to argue for adopting set theory as a ``foundation'' (a
``universal language'') for all of mathematics and to apply set
theoretic methods to algebraic number theory;\footnote{Most
spectacular of these was Cantor's proof that \textit{transcedental
numbers exist} by a counting argument much simpler than
Liouville's original proof---the first ``killer application'' of
set theory.} but this is not the main point---some people would
use category theory today and argue for its superiority over set
theory as both a foundation and a tool for applications. The main
point of looking for such ``definitions'' is to identify
\textit{the fundamental, characteristic properties of a
mathematical notion}, which, to repeat,  should be the
$\sim_P$-invariant properties of the set-theoretic objects that
model the notion.\smallskip

Another, fundamental (and much more sophisticated) example of this
process of constructing faithful modelings of mathematical notions
is the identification in~\cite{kolmogorov1933} of real-valued
\textit{random variables} with measurable functions
$X:\Omega\to\reals$ on a probability space, under various
equivalence relations.\ynm{equal, equal almost surely, or equal in distribution}

\section{Equational logic of  partial terms with conditionals}
\label{elptcond}

With the definitions in Section 1A of \ar, a \textit{partial function}
\[
f: X \pfto W
\]
with \textit{input set} $X$ and \textit{output
set} $W$ is an (ordinary) function $f:D_f\to W$ on some subset
$D_f\subseteq X$, the \textit{domain of convergence} of $f$. We
write
\[
f(x)\converges \eedf x\in D_f, \quad f(x)\diverges \eedf x\notin D_f
\qquad (x\in X)
\]
and most significantly, for $f,g : X\pfto W$ and $x\in X$,
\[
f(x) = g(x) \eedf \Big(f(x)\converges\conj g(x)\converges \conj f(x)=g(x)\Big)
\text{ or }\Big(f(x)\diverges \conj g(x)\diverges\Big).
\]
This relation between partial functions $f,g:X\pfto W$ with the
same input and output sets is sometimes called \textit{Kleene's
strong equality} between ``partial terms'', but there is nothing
partial about it: for any two $f,g: X\pfto W$ and any $x\in X$,
the proposition $f(x)=g(x)$ is either  true or false.

A \textit{signature} or \textit{vocabulary} is a set $\Phi$ of
\textit{function symbols}, each with assigned
$\arity(\phi)\in\nats$ and $\sort(\phi)\in\{\ind,\bool\}$; and as
in~\eqref{p-structure}, a (two-sorted, partial)
$\Phi$-\textit{structure} is a pair
\begin{equation}
\label{Phistr}
\A = (A, \{\phi^\A\}_{\phi\in\Phi})\quad(\phi\in\Phi, \phi^\A : A^{\arity(\phi)}\pfto A_{\sort(\phi)})
\end{equation}
of a (typically infinite) \textit{universe} $A$ and an
interpretation of the function symbols which matches their formal
arities and sorts, \ie for each $\phi\in\Phi$,
\[
\phi^\A : A^n \pfto A_s\qquad (n = \arity(\phi), s=\sort(\phi), A_\ind=A,
A_\bool=\{\true,\false\}).
\]

An \textit{expansion} of a $\Phi$-structure $\A$ in~\eqref{Phistr} is any $\Phi\cup\Psi$-structure
\begin{equation}
\label{expansion}
(\A,\Psi^\A)=(A,\{\phi^\A\}_{\phi\in\Phi}, \{\psi^\A\}_{\psi\in\Psi})
\end{equation}
on the same universe $A$ with $\Phi\cap\Psi=\emptyset$.

\ysection{Syntax}
\label{syntax}

The (pure, explicit) \textit{$\tPhi$-terms} are defined
by the \textit{structural recursion}
\begin{multline*}
F :\equiv  \true\mid \false \mid \f v_i \mid \f p^{s,n}(F_1,\ldots,F_n)\\
\mid \phi(F_1,\ldots,F_{\arity(\phi)})\mid
\tif F_0~\tthen F_1~\telse F_2,
\end{multline*}
where $\f v_0, \f v_1, \ldots$ is a fixed  list of
\textit{individual variables}; for each $s\in\{\ind,\bool\}$ and
each $n\in\nats$, $\f p^{s,n}_0, \f p^{s,n}_1, \ldots$ is a fixed
list of (partial) function variables of sort $s$ and arity $n$;
each term is assigned a sort in the natural way; and a
\textit{Parsing Lemma} (like Problem~x1E.1 of \ar) justifies the
standard method of defining functions on these terms by structural
recursion.\footnote{See Problem~\ref{fulltermdef} for a detailed
statement and proof of this. We just write $\phi$ for $\phi(~)$ when
$\arity(\phi)=0$, $\f p^{s,0}$ for $\f p^{s,0}(~)$  and we do not allow variables
over the set $\boolset=\{\false,\true\}$ of truth values, \cf
footnote 9 on p.\ 37 of \ar.}\smallskip

A $\Phi$-\textbf{term} is a $\tPhi$-term which has no function variables and a
$\Phi\cup\{\f p_1, \ldots, \f p_K\}$-\textbf{term} is a $\tPhi$-term whose function
variables are all in the list $\f p_1, \ldots, \f p_K$; these terms are
naturally interpreted in expansions $(\A,p_1,\ldots,p_K)$ of
$\Phi$-structures which interpret each $\f p_i$ by some $p_i$ of the correct sort and arity.\smallskip

An \textbf{extended} $\tPhi$-term is a pair
\[
(F, \vec{\f x}) \equiv F(\vec{\f x})
\equiv F(\f x_1,\ldots, \f x_n)
\]
of a $\tPhi$-term $F$ and a finite list of distinct
individual variables which includes all the  individual variables that occur
in $F$. The notation  provides a simple way to deal with substitutions,
\begin{multline*}
F(E_1,\ldots,E_n) :\equiv F\{\f x_1:\equiv E_1, \ldots, \f x_n :\equiv E_n\}\\
(F(\vec{\f x}) \text{ an extended term, $E_1,\ldots,E_n$ terms of sort $\ind$}).
\end{multline*}

\ysection{Semantics}

The \textit{denotation} (or \textit{value}) $\den((\A,\vec p),
F(\vec x))$ in a $\Phi$-structure $\A$ of each extended
$\Phi\cup\{\vec{\f p}\}$-term $F(\vec{\f x})$ for
given values $\vec p, \vec x$ of its variables is defined by the
usual (compositional) structural recursion on the definition of terms: skipping the
$(\A,\vec p)$ part (which remains constant),
\begin{gather*}
\den(\true(\vec x)) = \true,~~\den(\false(\vec x)) = \false, ~~\den(\f x_i(\vec x)) = x_i,
\\
\den(\f p_i(F_1,\ldots,F_n)(\vec x)) = p_i(\den(F_1(\vec x)), \ldots, \den(F_n(\vec x))),\\
\den(\phi(F_1,\ldots,F_{\arity(\phi)})(\vec x)) = \phi^\A(\den(F_1(\vec x)), \ldots,
\den(F_{\arity(\phi)}(\vec x))),\\
\den(\tif F_0~\tthen F_1~\telse F_2(\vec x))= \tif \den(F_0(\vec x))~\tthen
\den(F_1(\vec x))~\telse \den(F_2(\vec x));
\end{gather*}
and we will also use standard model-theoretic notation,\label{models}
\begin{align*}
(\A,\vec p) \models F(\vec x)=G(\vec x) &\eedf \den((\A,\vec p),F(\vec x)) = \den((\A,\vec p),G(\vec x)),\\
\A \models F(\vec{\f x})=G(\vec{\f x}) &\eedf
(\forall \vec p, \vec x)\Big((\A,\vec p) \models F(\vec x)=G(\vec x\Big)\\
\models F(\vec{\f x})=G(\vec{\f x})
&\eedf (\forall \A)\Big(\A\models F(\vec{\f x})=G(\vec{\f x})\Big).
\end{align*}

\problems

\begin{prob}
\label{fulltermdef}

A set $\mathcal T$ of pairs $(F,s)$ is closed under \textit{term formation} if
\begin{multline*}
(\true,\bool), (\false,\bool)\in\mathcal T,
\text{ for all }i,  (\f v_i,\ind)\in \mathcal T,\\
\text{for all }\phi\in\Phi, \Big((F_1,\ind), \ldots, (F_{\arity(\phi)}, \ind)\in\mathcal T\hspace*{3cm}\mbox{}\\
\mbox{}\hspace*{3cm}\implies (\phi(F_1,\ldots,F_{\arity(\phi)}),\sort(\phi))\in\mathcal T\Big),\\
\text{for all }\f p^{s,n}_i,
\Big((F_1,\ind), \ldots, (F_n, \ind)\in\mathcal T\hspace*{3cm}\mbox{}\\
\mbox{}\hspace*{3cm}\implies (\f p^{s,n}_i(F_1,\ldots,F_n),s)\in\mathcal T\Big)\\
\text{and }\Big((F_1,\bool), (F_2,s), (F_3,s)\in\mathcal T
\implies (\tif F_1~\tthen F_2~\telse F_3,s)\in \mathcal T\Big),
\end{multline*}
where we view $\true,\false$ and $\f v_i$ as \textit{strings of
symbols} of length $1$; and a string $F$ is a \textit{pure, explicit
$\tPhi$-term of sort $s$} if the pair $(F,s)$ belongs to every set
$\mathcal T$ which is closed under term formation.

Formulate a \textit{Parsing Lemma} for pure, explicit $\tPhi$-terms (as in Problem~x1E.1 of \ar\
for terms without function variables) and
give a complete proof of it using this precise definition.
\end{prob}

\section{Continuous, pure recursors}
\label{contrec}

For any two sets $X,W$, a (continuous, pure) \textit{recursor}
on $X$ to $W$ is a tuple
\begin{equation}
\label{recursor}
\alpha = (\alpha_0,\ldots,\alpha_K) : X\recto W,
\end{equation}
such that for suitable sets
$Y^\alpha_1,W^\alpha_1,\ldots,Y^\alpha_K,W^\alpha_K$,\footnote{For the definitions
of \textit{monotone} and \textit{continuous} functionals see Section~1A of \ar.}
\begin{align*}
\alpha_0 : X\times (Y^\alpha_1\pfto W^\alpha_1)\times\cdots\times
(Y^\alpha_K\pfto W^\alpha_K)&\pfto W, \text{ and}\\
\alpha_i :(Y^\alpha_1\pfto W^\alpha_1)\times\cdots\times
(Y^\alpha_K\pfto W^\alpha_K)&\pfto W^\alpha_i\quad(1\leq i\leq K)
\end{align*}
are continuous functionals. We allow
$X=\emptyproduct=\{\emptyset\}$ (as on page~\pageref{emptyproduct} of \ar),
in which case, by our conventions,
$\alpha_0 : (Y^\alpha_1\pfto W^\alpha_1)\times\cdots\times
(Y^\alpha_K\pfto W^\alpha_K)\pfto W$.

The number $K$ is the \textit{dimension} of $\alpha$, $\alpha_0$
is its \textit{output} or \textit{head functional}, the poset
$(Y_1\pfto W_1)\times\cdots\times (Y_K\pfto W_K)$ is its
\textit{solution space}, and we allow $K=0$--in which case there
is no solution space and $(\alpha_0)$ is naturally identified with
the partial function $\alpha_0:X\pfto W$. With the notation
of~\eqref{wheredef} of \ar, $\alpha$ \textit{defines} (or
\textit{computes}) the partial function $\oo\alpha:X\pfto W$,
where
\begin{equation}
\label{recursorvalue}
\oo\alpha(x) = \alpha_0(x,\vec p)
\twhere\Lbrace p_1(y_1) =
\alpha_1(y_1,\vec p),\ldots,p_K(y_K)=\alpha_K(y_K,\vec p)\Rbrace.
\end{equation}
We can summarize this situation succinctly by writing
\begin{multline}
\label{refwhereform}
\alpha(x)=(\alpha_0,\ldots,\alpha_K)(x)
\\=
\alpha_0(x,\vec p)
\rwhere\Lbrace p_1(y_1) =
\alpha_1(y_1,\vec p),\ldots,p_K(y_K)=\alpha_K(y_K,\vec p)\Rbrace,
\end{multline}
where ``\textit{where}'' is now an operation which assigns to every tuple
of continuous functionals $(\alpha_0,\ldots,\alpha_K)$ (with suitable input and
output sets) a recursor.

\ysection{The recursor defined by an extended program}
\label{recprogram}
Every (deterministic) extended recursive $\Phi$-program
\[
E(\vec{\f x})\equiv E_0(\vec{\f x})
\where\Lbrace\f p_1(\vec{\f x}_1)= E_1, \ldots,
\f p_K(\vec{\f x}_K)=E_K\Rbrace
\]
(as in~\eqref{recprogext}) with dimension $K$, arity $n$ and sort
$s$ defines naturally on each $\Phi$-structure $\A$ the continuous
pure \textit{recursor on $A$} of sort $s$ and arity $n$
\begin{equation}
\label{defprogrecursora}
\recursor(\A,E(\vec{\f x}))=(\alpha_0,\alpha_1,\ldots,\alpha_K)
: A^n\recto A_s,
\end{equation}
where
\begin{equation}
\label{defprogrecursorb}
\begin{array}{rcl}
\alpha_0(\vec x,\vec p)&=&\den((\A,\vec p),E_0(\vec x)),\\
\alpha_i(\vec x_i,\vec p)&=&\den((\A,\vec p),E_i(\vec x_i))
\quad(1\leq i\leq K)
\end{array}
\end{equation}
or, with the notation in~\eqref{refwhereform},
\begin{multline}
\label{defprogrecursor1}
\recursor(\A,E(\vec{x})) =
\den((\A,\vec p),E_0(\vec x))\\
\rwhere \Lbrace p_1(\vec x_1) = \den((\A,\vec p),E_1(\vec x_1)),\\
\ldots,p_K(\vec x_K) = \den((\A,\vec p),E_K(\vec x_K))\Rbrace.
\end{multline}

It is immediate from the semantics of recursive programs on
page~\pageref{progsemantics} of \ar\ that the recursor of an extended  program
computes its denotation,
\begin{equation}
\label{reccorrect}
\oo{\recursor(\A,E(\vec{x}))} = \den(\A,E(\vec x))
\quad (\vec x\in A^n).
\end{equation}
\noindent\textbf{However}
\label{however}: \textit{$\recursor(\A,E(\vec{\f x}))$ does not typically model faithfully the
algorithm expressed by $E(\vec{\f x})$ on $\A$}, partly because it does not
take into account the \textit{explicit steps} which may be
required to computes $\den(\A, E(\vec x))$ and (more importantly)
because {\it it does not record the dependence of that
algorithm on the primitives of $\A$.} In the extreme case, if
$E\equiv E_0\where\{~\}$ is a program with empty body (an explicit $\Phi$-term), then
\[
\recursor(\A,E(\vec{\f x})) = ((\lambda \vec x)\, \den(\A,E(\vec x)))
\]
is a trivial recursor of dimension $0$, a partial function on
$A$---and it is the same for all explicit terms which define this
partial function from any primitives, which is certainly not
right. In the next Section~\ref{cforms} we will define the
\textit{intension} $\refint(\A,E(\vec{\f x}))$ of $E(\vec{\f x})$
in $\A$ which (we will claim) models faithfully the algorithm from
the primitives of $\A$ \textit{expressed} by $E(\vec{\f x})$. As
it turns out, however,
\[
\refint(\A,E(\vec{\f x})) = \recursor(\A,\nf(E(\vec{\f x})))
\]
for some extended program $\nf(E(\vec{\f x}))$ which is
canonically associated with $E(\vec{\f x})$, and so every
recursive algorithm of a structure $\A$ is $\recursor(\A,F(\vec{\f
x})))$ for some $F(\vec{\f x})$.

\ynm{\ysection{About implementations (2)}
Our understanding of the algorithm
expressed by a recursor $\alpha$ in~\eqref{refwhereform} is that
it calls for solving the system of recursive equations in its body
and then plugging the solutions in the head functional to compute
the value $ \oo\alpha(x) = \alpha_0(x,\oo p_1,\ldots,\oo p_K)$.
\textit{How} we find the canonical solutions of this system is not
specified by $\alpha$: we might use the method outlined in the
proof of the Fixed Point Lemma~\ref{lfp}, or the recursive machine
defined in Section~\ref{cmodels} if $\alpha=\recursor(\A,E(\vec{\f
x}))$ is the recursor of an extended program in a structure, or
any one of several well-known and much studied
\textit{implementations} of recursion. What \textit{the
implementations of pure recursors} are is an important part of
this approach to the foundations of the theory of algorithms,
\cf~\cite{ynmsicily} and (especially) \cite{ynmpaschalis2008}
which includes a precise definition and some basic results. We
will not go into this here, except for a few additional comments
on page~\pageref{implementations3}.}

\ysection{Strong recursor isomorphism}
\label{recisomorphism}
The solutions of the system of recursive equations in the body of a recursor
$\alpha$ as in~\eqref{refwhereform} do not depend on the order in
which these equations are listed and the known methods for
computing them also do not depend on that order in any
important way; so it is natural to identify recursors which differ
only in the order in which their bodies are enumerated, so that,
for example
\begin{multline*}
\alpha_0(x,p_1,p_2,p_3)\rwhere\Lbrace p_1(y_1)=\alpha_1(y_1,p_1,p_2,p_3),\\
\hspace*{2cm}p_2(y_2)=\alpha_2(y_2,p_1,p_2,p_3),~~
p_3(y_3) = \alpha_3(y_3,p_1,p_2,p_3)\Rbrace\\
=
\alpha_0(x,p_1,p_2,p_3) \rwhere\Lbrace p_3(y_3) = \alpha_3(y_3,p_1,p_2,p_3),\\
\hspace*{1cm}p_1(y_1)=\alpha_1(y_1,p_1,p_2,p_3),~~
p_2(y_2)=\alpha_2(y_2,p_1,p_2,p_3)\Rbrace.
\end{multline*}
\ynm{Detail for the example
We need $q_1 = p_3, q_2 = p_1, q_3 = p_2$, so take $\pi(1)=3, \pi(2)=1,
\pi(3) = 2$, with the inverse being $\rho(3)=1, \rho(1) = 2, \rho(2)=3$.
This then requires the equations
\begin{align*}
\beta_0(x,q_1,q_2,q_3) &=
\alpha_0(x,p_1,p_2,p_3) =
\alpha_0(x,q_{\rho(1)},q_{\rho(2)},q_{\rho(3)})\\
\beta_1(y_3,q_1,q_2,q_3) &= \alpha_3(y_3,p_1,p_2,p_3)
= \alpha_3(y_3,q_{\rho(1)},q_{\rho(2)},q_{\rho(3)})
\end{align*}
For these we need to have
\[
Y^\beta_1 = Y^\alpha_3 = Y^\alpha_{\pi(1)},
Y^\beta_2 = Y^\alpha_1 = Y^\alpha_{\pi(2)},
Y^\beta_3 = Y^\alpha_2 = Y^\alpha_{\pi(3)}
\]
so
\[
q_i : Y^\beta_i=Y^\alpha_{\pi(i)}\pfto W^\alpha_{\pi(i)}=W^\beta_i,
\text{ so }q_{\rho(i)} : Y^\alpha_i\pfto W^\alpha_i.
\]}
The precise definition of this equivalence relation is
a bit messy in the general case:
two recursors $\alpha, \beta : X\iterto W$ on the same input
and output sets are \textit{strongly
isomorphic}---or just \textit{equal}---if they have
the same dimension $K$ and there is a permutation $\pi :
\{1,\ldots,K\}\bij \{1,\ldots,K\}$ with inverse $\rho=\pi^{-1}$
such that
\begin{equation}
\label{recidentity}
\begin{array}{c}
Y^\beta_i = Y^\alpha_{\pi(i)}, \quad
W^\beta_i= W^\alpha_{\pi(i)} \quad (i=1,\ldots,K),\\
\beta_0(x,q_1,\ldots,q_K) = \alpha_0(x,q_{\rho(1)}, \ldots, q_{\rho(K)}),\\
\beta_i(y_{\pi(i)},q_1,\ldots,q_K)
=\alpha_{\pi(i)}(y_{\pi(i)},q_{\rho(1)}, \ldots, q_{\rho(K)}),~
(1\leq i\leq K).
\end{array}
\end{equation}

Directly from the definitions, strongly
isomorphic recursors $\alpha,\beta : X\iterto W$
compute the same partial function $\oo{\alpha}=\oo{\beta} : X\pfto
W$; the idea, of course, is that \textit{strongly isomorphic
recursors model the same algorithm}, and so we will simply write
$\alpha=\beta$ to indicate that $\alpha$ and $\beta$ are strongly
isomorphic. This view is discussed in several of the articles
cited above and we will not go into it here.\footnote{This finest
relation of recursor equivalence was introduced (for
monotone recursors) in~\cite{ynmbotik}, and more general, weaker
notions were considered in~\cite{ynmsicily,ynmeng} and
in~\cite{ynmpaschalis2008}.}\smallskip

For the recursors defined by recursive programs, this definition takes
the following form, where for any $\tPhi$-term $F$ and
sequences of distinct function and individual variables $\vec{\f q}, \vec{\f p},
\vec{\f y}, \vec{\f x}$ (satisfying the obvious sort and arity conditions),
\begin{multline}
\label{alph}
F\{\vec{\f q}:\equiv \vec{\f p}, \vec{\f y}:\equiv \vec{\f  x}\}\\
:\equiv \text{the result of replacing in $F$ every $\f q_i$ by $\f p_i$
and every $\f y_j$ by $\f x_j$}.
\end{multline}

\begin{lemma}
\label{recprogiso}

Two extended $\Phi$-programs
\begin{equation}
\label{recprogisodisplay}
\left\{
\begin{array}{rcl}
E(\vec{\f x})&\equiv& E_0(\vec{\f x})
\where\Lbrace\f p_1(\vec{\f x}_1)= E_1, \ldots,
\f p_K(\vec{\f x}_K)=E_K\Rbrace,\\
F(\vec{\f x})&\equiv& F_0(\vec{\f x})
\where\Lbrace\f q_1(\vec{\f y}_1)= F_1, \ldots,
\f q_{K'}(\vec{\f y}_{K'})=E_{K'}\Rbrace
\end{array}
\right.
\end{equation}
of the same sort and arity define the same recursor on a
$\Phi$-structure $\A$ exactly when they have the same number of
parts $(K'=K)$ and there is a permutation
$\pi:\{0,\ldots,K\}\bij\{0,\ldots,K\}$ with inverse
$\rho=\pi^{-1}$ such that $\pi(0)=0$ and for all $\vec p$ and $\vec x$,
\begin{equation}
\label{recprogisochar}
\begin{array}{rcl}
(\A,\vec p) &\models& E_0(\vec x)=
F_0\{\vec{\f q} :\equiv \rho(\vec{\f p})\}(\vec x),\\
(\A,\vec p) &\models&
E_i(\vec x_i)=
F_{\pi(i)}\{\vec{\f q} :\equiv \rho(\vec{\f p}),\vec{\f y}_i:\equiv \vec{\f x}_i\}(\vec x_i), \quad (i=1,\ldots,K),
\end{array}
\end{equation}
where $\rho(\f p_1,\ldots,\f p_K)= (\f p_{\rho(1)},\ldots,\f p_{\rho(K)})$.
\end{lemma}

\begin{proofplain}[\sc The proof] is an exercise in managing
messy notations and we leave it for Problem~\ref{recprogisoprob}.
\end{proofplain}

\ysection{\tu{Extended} program congruence}

Two programs $E$ and $F$ are \textit{congruent} if $F$ can be
constructed from $E$ by an alphabetic change (renaming) of the
bound individual and function variables in its parts (as
in~\eqref{alph}) and a permutation of the equations in the body of
$E$. Congruence is obviously an equivalence relation on programs,
congruent programs have the same free variables and we write\label{congruence}
\begin{align*}
E\equiv_c F&\eedf \text{ $E$ and $F$ are congruent},\\
E(\vec{\f x})\equiv_cF(\vec{\f y})
&\eedf \vec{\f x}\equiv \vec{\f y} \text{ and }E\equiv_c F.
\end{align*}
By Lemma~\ref{recprogiso}, congruent extended programs define equal
recursors in every structure $\A$, \ie for every extended program
$E(\vec{\f x})$ and every program $F$,
\begin{equation}
\label{congintequal}
E\equiv_c F \implies(\forall \A)[
\recursor(\A,E(\vec{\f x}))=\recursor(\A,F(\vec{\f x}))],
\end{equation}
but the converse fails, \cf~Problem~\ref{congintequalprob2}.

\ysection{More general recursors}
\label{monrecursors}
The definition of recursors we gave is basically the one in
\cite{ynmbotik}, except that we allowed there the parts $\alpha_i$
of $\alpha$ to be (monotone but) not continuous and to depend on
the input set $X$, \ie
\begin{equation}
\tag{$\star$}
\label{botikrec}
\alpha = (\alpha_0,\alpha_1,\ldots,\alpha_K) : X\recto W,
\end{equation}
such that for suitable sets $Y^\alpha_1,W^\alpha_1,\ldots,Y^\alpha_K,W^\alpha_K$,
\begin{align*}
\alpha_0 : (Y^\alpha_1\pfto W^\alpha_1)\times\cdots\times
(Y^\alpha_K\pfto W^\alpha_K)\times X&\pfto W, \text{ and}\\
\alpha_i :Y^\alpha_i\times(Y^\alpha_1\pfto W^\alpha_1)\times\cdots\times
(Y^\alpha_K\pfto W^\alpha_K)\times X&\pfto W^\alpha_i\quad(1\leq i\leq K).
\end{align*}
Allowing the parts to be discontinuous is necessary for the theory
of higher-type recursion which was the topic
of~\cite{ynmflr,ynmbotik}, but their dependence on the input set
$X$ is not: the simpler, present notion is easier to work with and
it includes the more general objects if we identify
$\alpha$ in~\eqref{botikrec} with
\[
\alpha' = (\alpha'_0,\alpha'_1,\ldots,\alpha'_K) : X\recto W
\]
on the sets $Y'_i = X\times Y_i, W'_i=W_i$ ($1\leq
i \leq K$) with
\begin{align*}
\alpha'_0(x,p_1,\ldots,p_K) &=
\alpha_0(\lambda y_1p(x,y_1),
\ldots,\lambda y_Kp(x,y_K),x),\\
\alpha'_i(x,y_i,p_1,\ldots,p_K)
&=\alpha_i(\lambda y_1p(x,y_1), \ldots,\lambda y_Kp(x,y_K),x)
\quad(1\leq i\leq K).
\end{align*}

\hyphenation{Mos-cho-va-kis}

Substantially more general notions of (monotone  and continuous)
recursors whose solution spaces are products of arbitrary complete
posets were introduced in~\cite{ynmsicily,ynmeng} and (in greater
detail) in~\cite{ynmpaschalis2008}; and it is routine to define
\textit{nondeterministic algorithms} along the same lines, using
the fixpoint semantics of nondeterministic programs on
page~\pageref{ndfpsemantics} of \ar.\smallskip

\problems

\begin{prob}
\label{recprogisoprob}

Prove Lemma~\ref{recprogiso}. \hint Check by a (routine)
\looseness=+1
induction on terms, that {\it for any two lists of distinct function variables $\vec{\f
p}\equiv \f p_1, \ldots, \f p_K$ and $\vec{\f q}\equiv \f
q_1,\ldots, \f q_K$ and individual variables $\vec{\f y}\equiv \f
y_1,\ldots,\f y_m$, $\vec{\f x}\equiv \f x_1,\ldots,\f x_m$, and
for every $\Phi\cup\{\vec{\f q}\}$-term $M$ whose free variables
are all in the list $\vec{\f y}$,}
\begin{multline*}
\text{if }\beta(\vec y,\vec q)=\den((\A,\vec q),M(\vec y)),\\
\text{then }\beta(\vec x,\vec p)
=\den((\A,\vec p), M\{\vec{\f q}:\equiv\vec{\f p}, \vec{\f y}:\equiv
\vec{\f x}\}(\vec x)),
\end{multline*}
assuming, of course, that the sorts and arities of the function variables
are such that the claim makes sense.

\sol{The claim is obvious at the base $M\equiv \f y_i$ when
\[
\beta(\vec y,\vec q)=y_i, \quad
\den((\A,\vec p),\f y_i\{\vec{\f y}:\equiv \vec{\f x}\}(\vec x))=x_i
=\beta(\vec x, \vec p).
\]
In the induction step, the only non-trivial case is when
$M\equiv \f q_i(M_1,\ldots,M_n)$, and for this case, we compute
(taking $n=1$ for simplicity and skipping the lists $\vec{\f y}$,
$\vec{\f x}$ which do not enter the argument). We have
\[
\beta(\vec y,\vec q) = \den((\A,\vec q), \f q_i(M_1)(\vec y))
= q_i(\den((\A,\vec q),M_1(\vec y)))
= q_i(\beta_1(\vec y,\vec q)),
\]
so that by the induction hypothesis
\[
\beta_1(\vec x,\vec p) =
\den((\A,\vec p),M_1\{\vec{\f q}:\equiv \vec{\f p}\}(\vec x)),
\]
and
\begin{multline*}
\den((\A,\vec p), \f q_i(M_1)\{\vec{\f q}:\equiv \vec{\f p}\}(\vec x))
=
p_i(\den((\A,\vec p),M_1\{\vec{\f q}:\equiv \vec{\f p}\}(\vec x))
\\
= p_i(\beta_1(\vec x, \vec p))=\beta(\vec x,\vec p).
\end{multline*}

To prove the Lemma, we apply this fact to the parts $E_i$ and
$F_{\pi(i)}$ of the two programs and the lists of function variables
$\vec{\f p}$ and $\rho(\vec{\f q})$.}
\end{prob}

\begin{prob}
\label{congintequalprob2}
Prove that for every explicit term $E$ of sort $\bool$,
\[
\models \tif E~\tthen E~\telse E = E
\]
and use this to define two extended programs $E(\f x), F(\f x)$ such that
for all $\A$, $\recursor(\A,E(\vec{\f x}))=\recursor(\A,F(\vec{\f x}))$
but $E\not\equiv_c F$.

\sol{The identity is trivial, taking cases for each $\sigma$ on
whether $\sigma(\A,E)\diverges$, $\sigma(\A,E)=\true$, or
$\sigma(\A,E)=\false$. For the counterexample, on any structure, take
\[
E\equiv\true\where\Lbrace p = p\Rbrace,
\quad
F\equiv \true\where\Lbrace p = \tif p~\tthen p~\telse p\Rbrace.
\]
}
\end{prob}

\section{Canonical forms and intensions}
\label{cforms}

In this section we will associate with each extended (recursive) $\Phi$-program
\begin{equation}
\label{recprog}
E(\vec{\f x}) \equiv E_0(\vec{\f x}) \where\Lbrace\f p_1(\vec{\f x}_1)= E_1, \ldots,
\f p_K(\vec{\f x}_K)=E_K\Rbrace
\end{equation}
a (unique up to congruence) \textit{canonical form} $\nf(E(\vec{\f x}))$,
so that the construction
\[
(\A,E(\vec{\f x}))\mapsto
\recursor(\A,\nf(E(\vec{\f x})))
\]
captures the algorithm expressed by $E(\vec{\f x})$
in any $\Phi$-structure $\A$. This theory of canonical
forms yields, in particular, a robust notion of \textit{the
elementary \tu{first-order} algorithms} of a structure $\A$.

\ysection{Some more syntax}

For the syntactic work we will do in this Section we need to enrich
and simplify the notation of \ar\ summarized in Section~\ref{elptcond}.

We introduce two function symbols $\ycond_s$ of arity $3$ and
sort $s$ ($\equiv\ind$ or $\bool$) and the
abbreviations
\[
\ycond_s(F_0,F_1,F_2) :\equiv \tif F_0 ~\tthen F_1~\telse F_2;
\]
this is semantically inappropriate because the conditionals do not define
(strict) partial functions, but it simplifies considerably the definition
of pure, explicit $\tPhi$-terms on page \pageref{syntax}
which now takes the form
\begin{equation}
\tag{Pure explicit terms}
F :\equiv \true \mid \false \mid \f v_i \mid \f c(F_1,\ldots,F_n),
\end{equation}
where $\f c$ is any $n$-ary function symbol $\phi$ in the
vocabulary $\Phi$, or an $n$-ary function variable $\f p^{s,n}_i$
or $\ycond_s$ (and $n=3$).\smallskip

\ysection{Set representations}
With each extended $\Phi$-program $E(\vec{\f x})$ as in~\eqref{recprog}, we associate its
\textit{set representation}
\begin{equation}
\label{setrep}
\set(E(\vec{\f x}))=
\Lbrace E_0(\vec{\f x}),~~\f p_1(\vec{\f x}_1)= E_1, \ldots,
~~\f p_K(\vec{\f x}_K)=E_K\Rbrace.
\end{equation}
Notice that $\set(E(\vec{\f x}))$ determines $E_0(\vec{\f x})$, its only member which is not an equation,
and by the definition of the congruence relation on  page~\pageref{congruence},
\[
\set(E(\vec{\f x}))=\set(F(\vec{\f x})) \implies E\equiv_c F;
\]
the converse implication fails, but the algorithm which computes
canonical forms naturally operates on these set representations of
extended terms, so having a notation for them simplifies greatly its
specification.

\ysection{Immediate and irreducible terms}

A term $I$ is \textit{immediate} if it is an individual variable
or a direct call to a recursive variable applied to individual
variables,\ynm{Checked this out in detail.}
\begin{equation}
\tag{Immediate terms}
\label{immediate}
I :\equiv \f v_i \mid \f p^{s,0} \mid \f p^{n,s}_i(\f u_1,\ldots,\f u_n);
\end{equation}
and a term $F$ is \textit{irreducible} if it is immediate, a Boolean
constant, or of the form $\f c(I_1,\ldots,I_l)$ with immediate
$I_1,\ldots, I_n$,\footnote{There are natural sort
restrictions in these definitions and assignments of sorts to
immediate and irreducible terms that we will suppress in the
notation, \cf Problem~\ref{immsorts}.}
\begin{equation}
\tag{Irreducible terms}
\label{irreducible}
F :\equiv I \mid \true\mid\false \mid \f c(I_1,\ldots,I_n)
\end{equation}
For example, $\f u, \f p(\f u_1,\f u_2,\f u_1)$ are immediate,
$\phi(\f u)$ and $\f p(\f u,\f q(\f v),\f r(\f u))$ are irreducible but not
immediate, and $\f p(u,\phi(\f v))$ is reducible.

Computationally, we think of immediate terms as ``generalized
variables'' which can be accessed directly, like array entries
$a[i], b[i,j,i]$ in some programming languages, and a term is
irreducible if it can be computed by direct (not nested) calls to primitives or
to recursive variables. Both of these properties of terms are
(trivially) preserved by alphabetic change of
variables~\eqref{alph},
\begin{multline}
\label{alphchange}
\textit{if $F$ is immediate \ulp irreducible\urp},\\
\textit{then $F\{\vec{\f q}:\equiv \vec{\f p}, \vec{\f y}:\equiv \vec{\f x}\}$
is also immediate \ulp irreducible\urp}.
\end{multline}

\ysection{Irreducible extended programs}

An extended program
\[
E(\vec{\f x}) \equiv E_0(\vec{\f x})\where\Lbrace\f p_1(\vec{\f x}_1)= E_1, \ldots,
\f p_K(\vec{\f x}_K)=E_K\Rbrace
\]
is \textit{irreducible} if $E_0, E_1,\ldots, E_K$ are all
irreducible terms. These are the extended programs which cannot be
reduced by the basic process of \textit{reduction}, to which we
turn next.

\ysection{Arrow reduction}
\label{reduction}

We first define the simplest reduction relation
\[
E(\vec{\f x})\xrightarrow{(\f p,j,\f q)} F(\vec{\f x}),
\quad \vec{\f x} \equiv (\f x_1, \ldots, \f x_n)
\]
on extended programs with the same list of $n$ free variable,
where $\f p, \f q$ are function variables and $j$ is a number. The
definition splits in two cases on the triple $(E,\f p,j)$ and it
helps to call a function variable $\f r$ \textit{fresh} if it is none of the function
variables $\f p_i$ in  the body of $E$.\smallskip

(1) \textit{The body case}: {\it one of the equations in the body of $E$ is
\[
\f p(\vec{\f x}_{\f p})
= \f c(G_1,\ldots,G_{j-1},G_j,G_{j+1},\ldots, G_l)
\]
and the term $G_j$ is not immediate}. Put
\begin{multline*}
E(\vec{\f x})\xrightarrow{(\f p,j,\f q)} F(\vec{\f x})
\eedf \text{$\f q$ is fresh, $\arity(\f q)=\arity(\f p)$, $\sort(\f q) = \sort(\f p)$, and}\\
\set(F(\vec{\f x})) = \Big(\set(E(\vec{\f x}))\setminus \Lbrace\f p(\vec{\f x}_{\f p}) =
\f c(G_1,\ldots,G_{j-1},G_j,G_{j+1},\ldots,G_l)\Rbrace\Big)\\
\cup\Lbrace
\f p(\vec{\f x}_{\f p})
= \f c(G_1,\ldots,G_{j-1},\f q(\vec{\f x}_{\f p}),G_{j+1},\ldots, G_l),\quad
\f q(\vec{\f x}_{\f p}) = G_j\Rbrace.
\end{multline*}

(2) \textit{The head case}: {\it $\f p$ and $\f q$ are both fresh,
$\arity(\f q)=n$, $\sort(\f q)=\sort(E_0)$},
\[
E_0\equiv\f c(G_1,\ldots,G_{j-1},G_j,G_{j+1},\ldots,G_l)
\]
{\it and  the term $G_j$ is not immediate}. Put
\begin{multline*}
E(\vec{\f x}) \xrightarrow{(\f p,j,\f q)} F(\vec{\f x})\\
\eedf
\set(F(\vec{\f x})) = \Big(
\set(E(\vec{\f x}))\setminus\Lbrace\f c(G_1,\ldots,G_{j-1},G_j,G_{j+1},\ldots,G_l)\Rbrace\Big)\\
\cup\Lbrace \f c(G_1,\ldots,G_{j-1},\f q(\vec{\f x}),G_{j+1},\ldots, G_l), \quad
\f q(\vec{\f x}) = G_j\Rbrace.
\end{multline*}

Notice  that in the head case, the specific function variable $\f
p$ does not occur on the right-hand-side of the definition as it
does in the body case, it is used only as a ``marker'' that this
is the head case: we put
\begin{multline}
\label{marker}
\f p\sim_E\f p' \eedf \text{either }\f p\equiv\f p'\in \{\f p_1,\ldots,\f p_K\}
\quad \text{(the body case)}\\
\text{ or }
\{\f p, \f p'\}\cap\{\f p_1,\ldots,\f p_K\}=\emptyset\quad \text{(the head case)}.
\end{multline}

\begin{figure}[t]
\[
\begin{array}{cc}
\begin{array}{cclcl}
E(\vec{\f x})&:&
\begin{array}{rcl}
\f p(\vec{\f x}_{\f p}) &=& \f c(G_1,G_2,G_3)
\end{array}\\
&&\qquad\xrightarrow{(\f p,1, \f q)}
\\
F(\vec{\f x})&:&
\begin{array}{rcl}
\f p(\vec{\f x}_{\f p}) &=& \f c(\f q(\vec{\f x}_{\f p}),G_2,G_3)\\
\f q(\vec{\f x}_{\f p}) &=& G_1
\end{array}\\
&&\qquad\xrightarrow{(\f p, 3, \f q')}
\\
 H(\vec{\f x})&:&
\begin{array}{rcl}
\f p(\vec{\f x}_{\f p}) &=& \f c(\f q(\vec{\f x}_{\f p}),G_2,\f q'(\vec{\f x}_{\f p}))\\
\f q'(\vec{\f x}_{\f p}) &=& G_3\\
\f q(\vec{\f x}_{\f p}) &=& G_1
\end{array}
\end{array}
\quad&
\begin{array}{cclcl}
E(\vec{\f x})&:&
\begin{array}{rcl}
\f p(\vec{\f x}_{\f p}) &=& \f c(G_1,G_2,G_3)
\end{array}\\
&&\qquad\xrightarrow{(\f p,3, \f q')}
\\
F'(\vec{\f x})&:&
\begin{array}{rcl}
\f p(\vec{\f x}_{\f p}) &=& \f c(G_1,G_2,\f q'(\vec{\f x}_{\f p}))\\
\f q'(\vec{\f x}_{\f p}) &=& G_3
\end{array}\\
&&\qquad\xrightarrow{(\f p, 1, \f q)}
\\
H(\vec{\f x})&:&
\begin{array}{rcl}
\f p(\vec{\f x}_{\f p}) &=& c(\f q(\vec{\f x}_{\f p}),G_2,\f q'(\vec{\f x}_{\f p}))\\
\f q(\vec{\f x}_{\f p}) &=& G_1\\
\f q'(\vec{\f x}_{\f p}) &=& G_3
\end{array}
\end{array}
\end{array}
\]
\vspace*{-10pt}
\caption{Amalgamation example, the body case.}
\label{amalgbody}
\end{figure}

\begin{lemma}
\label{redfree}

\tu{1} If\/ $E(\vec{\f x}) \xrightarrow{(\f p,j,\f q)} F(\vec{\f x})$
and $E(\vec{\f x}) \xrightarrow{(\f p,j,\f q)} F'(\vec{\f x})$,
then $F\equiv_c F'$.\smallskip

\tu{2} If\/ $\f q'$ is fresh, $\arity(\f q')=\arity(\f q)$ and $\sort(\f q')=\sort(\f q)$, then
\begin{equation}
\label{redfree2}
E(\vec{\f x}) \xrightarrow{(\f p,j,\f q)} F(\vec{\f x})
 \implies E(\vec{\f x}) \xrightarrow{(\f p,j,\f q')} F\{\f q :\equiv \f q'\}(\vec{\f x}).
\end{equation}

\tu{3} If\/ $\f r'$ is fresh and $\f r'\not\equiv\f q$, then
\begin{equation}
\label{redfree3}
E(\vec{\f x}) \xrightarrow{(\f p,j,\f q)} F(\vec{\f x})
\implies E\{\f r:\equiv \f r'\}(\vec{\f x}) \xrightarrow{(\f p,j,\f q)} F\{\f r:\equiv \f r'\}(\vec{\f x}).
\end{equation}

\tu{4} An extended program $E(\vec{\f x})$ is irreducible if and only if there are
no $\f p,j,\f q, F$ such that
$E(\vec{\f x}) \xrightarrow{(\f p,j,\f q)} F(\vec{\f x})$.\smallskip

\tu{5} For every extended program $E(\vec{\f x})$, every program $F$, every triple
$(\f p,j,\f q)$ and every structure $\A$,
\[
E(\vec{\f x})\xrightarrow{(\f p,j,\f q)}F(\vec{\f x}) \implies
\den(\A,E(\vec{x}))=\den(\A,F(\vec{x}))\quad(\vec x\in A^n).
\]
\end{lemma}

\begin{proof}
\ynm{If, for example, the body case applies and $\f q$ is any
variable with the correct arity and sort, we can take
\begin{multline*}
F \equiv E_0\where\Lbrace\f p_1(\vec{\f x}_1)= E_1, \ldots,\f p_{j-1}(\f x_{j-1}),
\f p_{j-1}(\f x_{j-1}),\ldots,\\
 \f p_K(\vec{\f x}_K)=E_K,
\f p(\vec{\f x}_{\f p})
= \f c(G_1,\ldots,G_{j-1},\f q(\vec{\f x}_{\f p}),G_{j+1},\ldots, G_l),~~
\f q(\vec{\f x}_{\f p}) = G_j\Rbrace.
\end{multline*}}
(1)~--~(3) are trivial, once we get through the notation and (4) follows immediately
by the definition. (5) is also quite easy, either by a direct, fixpoint argument or by applying the
\textit{Head} and \textit{Beki\v c-Scott} rules of Theorem~\ref{recidentities} of
\ar.
\end{proof}

\begin{figure}[t]
\[
\begin{array}{cc}
\begin{array}{cclcl}
E(\vec{\f x})&:&
\begin{array}{rcl}
 && \f c(G_1,G_2,G_3)
\end{array}\\
&&\qquad\xrightarrow{(\f p,1, \f q)}
\\
F(\vec{\f x})&:&
\begin{array}{rcl}
 \f c(\f q(\vec{\f x}),G_2,G_3)\\
\f q(\vec{\f x}) = G_1
\end{array}\\
&&\qquad\xrightarrow{(\f p, 3, \f q')}
\\
 H(\vec{\f x})&:&
\begin{array}{rcl}
\f c(\f q(\vec{\f x}),G_2,\f q'(\vec{\f x}))\\
\f q'(\vec{\f x}) = G_3\\
\f q(\vec{\f x}) = G_1
\end{array}
\end{array}
\qquad&
\begin{array}{cclcl}
E(\vec{\f x})&:&
\begin{array}{rcl}
\f c(G_1,G_2,G_3)
\end{array}\\
&&\qquad\xrightarrow{(\f p,3, \f q')}
\\
F'(\vec{\f x})&:&
\begin{array}{rcl}
\f c(G_1,G_2,\f q'(\vec{\f x}))\\
\f q'(\vec{\f x}) = G_3
\end{array}\\
&&\qquad\xrightarrow{(\f p, 1, \f q)}
\\
H(\vec{\f x})&:&
\begin{array}{rcl}
c(\f q(\vec{\f x}),G_2,\f q'(\vec{\f x})\\
\f q(\vec{\f x}) = G_1\\
\f q'(\vec{\f x}) = G_3
\end{array}
\end{array}
\end{array}
\]
\vspace*{-10pt}
\caption{Amalgamation example, the head case.}
\label{amalghead}
\end{figure}

The key property of the reduction calculus is the following, simple

\begin{lemma}[Amalgamation] If $(\f p\not\sim_E \f p' \text{ or } j\neq j'), ~~\f q\not\equiv\f q'$,
\label{perm}
\begin{equation}
\label{permhyp}
E(\vec{\f x})\xrightarrow{(\f p,j,\f q)} F(\vec{\f x})
\text{ and } E(\vec{\f x})\xrightarrow{(\f p',j', \f q')} F'(\vec{\f x}),
\end{equation}
then there is a program $H$ such that
\begin{equation}
\label{permconclusion}
E(\vec{\f x})\xrightarrow{(\f p,j,\f q)} F(\vec{\f x})\xrightarrow{(\f p',j',\f q')} H(\vec{\f x})
\text{ and }
E(\vec{\f x})\xrightarrow{(\f p',j',\f q')} F'(\vec{\f x})\xrightarrow{(\f p,j,\f p)} H(\vec{\f x}).
\end{equation}
\end{lemma}

\begin{proof}

This is obvious if $\f p\not\sim_E\f p'$, so that the two assumed
reductions operate independently on different parts of $E$ and
they give the same result if they are executed in either order.
For the proof in the more interesting case when $\f p\sim_E \f
p'$, we just put down as examples the relevant sequences of
needed replacements in Diagram~\ref{amalgbody} for the body case when $\f
p$ is ternary, $j=1$ and $j'=3$, and in Diagram~\ref{amalghead}
for the head case when $j=1$ and $j'=3$.  The general case is only
notationally more complex.
\end{proof}

\ysection{Extended program reduction}
Using now arrow reduction,
we define the \textbf{one-step} and \textbf{full} reduction
relations on extended programs by \begin{gather} \label{onestep}
E(\vec{\f x})\redto_1 F(\vec{\f x}) \eedf \text{ for some }\f p,
j, \f q,
E(\vec{\f x})\xrightarrow{(\f p,j,\f q)} F(\vec{\f x}),\\
\label{fullred}
E(\vec{\f x})\redto F(\vec{\f x}) \eedf E\equiv_c F\hspace*{5cm}\\\nonumber
\hspace*{1cm} \text{ or }(\exists F^1,\ldots,F^k)[
E(\vec{\f x})\redto_1 F^1(\vec{\f x})\redto_1\cdots\redto_1 F^k(\vec{\f x})
\text{ and }F^k\equiv_c F].
\end{gather}

It is immediate from Lemma~\ref{redfree} that for every structure $\A$, every extended
program $E(\vec{\f x})$ and every $F$,
\begin{equation}
\label{sameden}
E(\vec{\f x})\redto F(\vec{\f x}) \implies \den(\A,E(\vec{x}))=\den(\A,F(\vec{x}))\quad(\vec x\in A^n),
\end{equation}
but this is true even if we removed the all-important,
non-immediacy restrictions in the definition of reduction, by
Theorem~\ref{recidentities} of \ar. The claim is that
\begin{multline*}
\textit{if }E(\vec{\f x})\redto F(\vec{\f x}),\\
\textit{then $E(\vec{\f x})$ and $F(\vec{\f x})$ express the same
algorithm in every $\Phi$-structure},
\end{multline*}
but we will not get here into defending this naive view beyond
what we said in Sections~\ref{algorithmssec} and~\ref{contrec}.\smallskip

\textbf{Caution}. Reduction is a syntactic operation on extended recursive programs
which models (very abstractly) \textit{partial compilation},
bringing the mutual recursion expressed by the program to a useful
form before the recursion is implemented \textit{without
committing to any particular method of implementation}.
No real computation is done by it, and it does not
embody any ``optimization'' of the algorithm expressed by an
extended program, \cf Problem~\ref{noopt}.\smallskip

\ysection{Size} Let $\size(E)$ be \textit{the number of
occurrences of non-im\-me\-di\-ate proper subterms}\footnote{See
Problem~\ref{size0} for a rigorous definition and proofs of the
properties of $\size(E)$.} of some part of a program
$E$. For example, if
\begin{multline}
\label{expexample}
E \equiv\f p(\f x,\boxed{\phi_0(\f x)})
\where \\
\Lbrace \f p(\f x, \f y) = \tif \boxed{\text{test}(\boxed{\phi_0(\f x)},\f y)}
~\tthen \f y ~\telse \boxed{\f p(\f x,\boxed{\sigma(\f x,\f y)})}\Rbrace,
\end{multline}
with $\Phi=\{\phi_0,\textup{test},\sigma\}$,
then $\size(E)=5$, because   the proper subterms we count are the
boxed $\phi_0(\f x)$ (twice), $\text{test}(\phi_0(\f x),\f y)$,
$\sigma(\f x,\f y)$, and $\f p(\phi_0(\f x),\sigma(\f x,\f
y))$.\smallskip

From the definition of reduction, easily, an extended program $E(\vec{\f x})$
is irreducible exactly when $\size(E)=0$ and each one-step reduction lowers size by $1$,
so, trivially:

\begin{lemma}
\label{redlength}
If\/ $E(\vec{\f x})\redto_1 F^1(\vec{\f x})\redto_1 \cdots \redto_1 F^k(\vec{\f x})$ is
a sequence of one-step reductions starting with $E(\vec{\f x})$, then $k\leq\size(E)$;
and $F^k(\vec{\f x})$ is irreducible if and only if\/ $k=\size(E)$.
\end{lemma}

We illustrate the reduction process by constructing in
Figure~\ref{redexample} a maximal reduction sequence starting with $E(\f x)$
with $E$ the program in~\eqref{expexample}.

\ynm{\f p(\f x, \f y) = \tif \text{test}(\phi_0(\f x),\f y)
~\tthen \f y ~\telse \f p(\phi_0(\f x),\sigma(\f x,\f y))
}
\begin{figure}[t]
\gdef\Lbrace{\{}
\gdef\Rbrace{\}}
\[
\begin{array}{rcl}
E(\f x)&:&
\begin{array}{rcl}
\f p(\f x,\phi_0(\f x))&&\hspace*{-15pt}\where\Lbrace\\
\f p(\f x,\f y) &=&\tif \underline{\text{test}(\phi_0(\f x),\f y)}
~\tthen \f y ~\telse \f p(\f x,\sigma(\f x,\f y))\Rbrace
\end{array}\\
&&\qquad\xrightarrow{(\f p,1, \f q_1)~}
\\
F^1(\f x)&:&
\begin{array}{rcl}
\f p(\f x, \phi_0(\f x))&&\hspace*{-15pt}\where\Lbrace\\
\f p(\f x,\f y) &=& \tif \f q_1(\f x, \f y)~\tthen \f y ~\telse \f p(\f x,\underline{\sigma(\f x,\f y)}),\\
\f q_1(\f x,\f y) &=&\text{test}(\phi_0(\f x),\f y)\Rbrace
\end{array}\\
&&\qquad\xrightarrow{(\f p, 3, \f q_2)~}
\\
F^2(\f x)&:&
\begin{array}{rcl}
\f p(\f x, \phi_0(\f x))&&\hspace*{-15pt}\where\Lbrace\\
\f p(\f x,\f y) &=& \tif \f q_1(\f x, \f y)~\tthen \f y ~\telse \f q_2(\f x,\f y),\\
\f q_2(\f x,\f y) &=& \f p(\underline{\phi_0(\f x)},\sigma(\f x,\f y)),
\f q_1(\f x,\f y) = \text{test}(\phi_0(\f x),\f y)\Rbrace
\end{array}\\
&&\qquad\xrightarrow{(\f q_2,1,q_3)~}
\\
F^3(\f x)&:&
\begin{array}{rcl}
\f p(\f x, \phi_0(\f x))&&\hspace*{-15pt}\where\Lbrace\\
\f p(\f x,\f y) &=&  \tif \f q_1(\f x, \f y)~\tthen \f y ~\telse \f q_2(\f x,\f y),\\
\f q_2(\f x,\f y) &=& \f p(\f q_3(\f x, \f y),\sigma(\f x,\f y)),
\f q_3(\f x,\f y) = \phi_0(\f x),\\
\f q_1(\f x,\f y) &=&\text{test}(\underline{\phi_0(\f x)},\f y)\Rbrace
\end{array}
\\
&&\qquad\xrightarrow{(\f q_1,1,\f q_4)~}\\
F^4(\f x)&:&
\begin{array}{rcl}
\f p(\f x, \phi_0(\f x))&&\hspace*{-15pt}\where\Lbrace\\
\f p(\f x,\f y) &=&  \tif \f q_1(\f x, \f y)~\tthen \f y ~\telse \f q_2(\f x,\f y),\\
\f q_2(\f x,\f y) &=& \f p(\f q_3(\f x, \f y),\underline{\sigma(\f x,\f y)}),\\
\f q_3(\f x,\f y) &=& \phi_0(\f x),\\
\f q_1(\f x,\f y) &=&\text{test}(\f q_4(\f x, \f y),\f y),\\
\f q_4(\f x, \f y) &=& \phi_0(\f x)\Rbrace
\end{array}
\\
&&\qquad\xrightarrow{(\f q_2,2,\f q_5)~}\\
F^5(\f x)&:&
\begin{array}{rcl}
\f p(\f x, \underline{\phi_0(\f x)})&&\hspace*{-15pt}\where\Lbrace\\
\f p(\f x,\f y) &=&  \tif \f q_1(\f x, \f y)~\tthen \f y ~\telse \f q_2(\f x,\f y),\\
\f q_2(\f x,\f y) &=& \f p(\f q_3(\f x, \f y),\f q_5(\f x, \f y)),\\
\f q_5(\f x, \f y) &=& \sigma(\f x, \f y),\\
\f q_3(\f x,\f y) &=& \phi_0(\f x), ~~\f q_1(\f x,\f y) =\text{test}(\f q_4(\f x, \f y),\f y),\\
\f q_4(\f x, \f y) &=& \phi_0(\f x)\Rbrace
\end{array}
\\
&&\qquad\xrightarrow{(\f r,2,\f q_6)~}\\
F^6(\f x)&:&
\begin{array}{rcl}
\f p(\f x, \f q_6(x))&&\hspace*{-15pt}\where\Lbrace\\
\f p(\f x,\f y) &=&  \tif \f q_1(\f x, \f y)~\tthen \f y ~\telse \f q_2(\f x,\f y),\\
\f q_2(\f x,\f y) &=& \f p(\f q_3(\f x, \f y),\f q_5(\f x, \f y)),\\
\f q_5(\f x, \f y) &=& \sigma(\f x, \f y),\\
\f q_3(\f x,\f y) &=& \phi_0(\f x), ~~\f q_1(\f x,\f y) =\text{test}(\f q_4(\f x, \f y),\f y,)\\
\f q_4(\f x, \f y) &=& \phi_0(\f x), \f q_6(x) = \phi_0(\f x)\Rbrace
\end{array}
\end{array}
\]
\vspace*{-10pt}
\caption{A maximal reduction sequence for $E(\f x)$ in~\eqref{expexample}.}
\label{redexample}
\end{figure}

\begin{lemma}
\label{cfthmlemma}
For any extended program $E(\vec{\f x})$ and irreducible $F^1(\vec{\f x}), F^2(\vec{\f x})$,
\[
\Big(E(\vec{\f x})\redto F^1(\vec{\f x}) \conj E(\vec{\f x})\redto F^2(\vec{\f x})\Big)
\implies F^1\equiv_c F^2.
\]
\end{lemma}

\begin{proofplain}
is by induction on $\size(E)$.\smallskip

\textit{Basis}, $\size(E)\leq 1$. If $\size(E)=0$ so that $E(\vec{\f x})$ is
irreducible, then the hypothesis gives immediately $E\equiv_c F^1$ and
$E\equiv_c F^2$, so $F^1\equiv_c F^2$.\smallskip

If $\size(E)=1$, then there is exactly one part $E_i$ of $E$ which
can activate a one-step reduction. In the head case, the definition gives
\begin{multline*}
\set(F^1(\vec{\f x})) = \Big(
\set(E(\vec{\f x}))\setminus\Lbrace\f c(G_1,\ldots,G_{j-1},G_j,G_{j+1},\ldots,G_l)\Rbrace\Big)\\
\cup\Lbrace \f c(G_1,\ldots,G_{j-1},\f q(\vec{\f x}),G_{j+1},\ldots, G_l), ~~
\f q(\vec{\f x}) = G_j\Rbrace
\end{multline*}
with a fresh function variable $\f q$ and the same equation for
$\set(F^2(\vec{\f x}))$ with a (possibly) different fresh function variable
$\f q'$; but then $F^1\equiv_c F^2$ by~(1) of
Lemma~\ref{redfree}. The same argument works in the body case.\smallskip

\textit{Induction Step}, $k=\size(E)\geq 2$. The hypothesis and
Lemma~\ref{redlength} supply reduction sequences
\begin{multline}
\tag{$\ast$}
\label{star1}
E(\vec{\f x})\redto_1 F^{1,1}(\vec{\f x}) \redto_1 \cdots \redto_1F^{1,k}(\vec{\f x})\equiv F^1(\vec{\f x}),
\\
E(\vec{\f x})\redto_1 F^{2,1}(\vec{\f x}) \redto_1 \cdots \redto_1F^{2,k}(\vec{\f x})\equiv F^2(\vec{\f x}),
\end{multline}
so focussing on the first one-step reductions in the two hypotheses,
there are triples $(\f p,j,\f q)$ and $(\f p',j',\f q')$ such that
\begin{equation}
\label{keyconcl}
E(\vec{\f x})\xrightarrow{(\f p,j,\f q)}F^{1,1}(\vec{\f x})
\text{ and }E(\vec{\f x})\xrightarrow{(\f p',j',\f q')} F^{2,1}(\vec{\f x}).
\end{equation}
We consider three cases on how this may arise:\smallskip

\textit{Case 1}: $(\f p\not\sim_E \f p' \text{ or } j\neq j')$ and $\f q\not\equiv\f q'$.
The Amalgamation Lemma~~\ref{perm} applies in this case and supplies a program $H$ such that
\begin{equation}
\tag{$\ast\ast$}
F^{1,1}(\vec{\f x})\redto_1 H(\vec{\f x}) \text{ and }F^{2,1}(\vec{\f x})\redto_1 H(\vec{\f x}).
\end{equation}
We fix a (maximal) reduction sequence
\begin{equation}
\tag{$\ast\ast\ast$}
H(\vec{\f x}) \redto_1 H^3(\vec{\f x})\redto_1\cdots\redto_1 H^k(\vec{\f x})
\text{ (with an irreducible $H^k(\vec{\f x})$)}
\end{equation}
and notice that from the reductions in the three starred displays,
\[
F^{1,1}(\vec{\f x})\redto F^1(\vec{\f x}),~~ F^{1,1}(\vec{\f x})\redto H^k(\vec{\f x}),
\qquad F^{2,1}(\vec{\f x})\redto F^2(\vec{\f x}),~~ F^{2,1}(\vec{\f x})\redto H^k(\vec{\f x});
\]
but $\size(F^{1,1})=\size(F^{2,1})=k-1$, so the Induction Hypothesis applies and yields the
required $F^1\equiv_c H^k \equiv_c F^2$.\smallskip

\textit{Case 2}, $(\f p\not\sim_E \f p' \text{ or } j\neq j')$ but $\f q\equiv\f q'$, so the first
one-step reductions supplied by the Hypothesis are
\[
E(\vec{\f x})\xrightarrow{(\f p,j,\f q)}F^{1,1}(\vec{\f x})
\text{ and }E(\vec{\f x})\xrightarrow{(\f p',j',\f q)} F^{2,1}(\vec{\f x}).
\]
If $\f q'$ is fresh but with the same arity and sort as $\f q$, then (2) of Lemma~\ref{redfree} gives
\[
E(\vec{\f x})\xrightarrow{(\f p',j',\f q')} F^{2,1}\{\f q:\equiv \f q'\}(\vec{\f x});
\]
now (3) of Lemma~\ref{redfree} gives
\[
F^{2,1}\{\f q:\equiv \f q'\}(\vec{\f x}) \redto_1 \cdots \redto_1 F^{2,k}\{\f q:\equiv \f q'\}(\vec{\f x})
\]
so that Case 1 applies and gives
\[
F^1\equiv_c F^2\{\f q:\equiv \f q'\};
\]
but $F^2\{\f q:\equiv \f q'\}\equiv_c F^2$, which completes the argument in this case.\smallskip

\textit{Case 3}, $\f p\sim\f p', j=j'$ and $\f q\not\equiv \f q'$. The first one-step reductions
supplied by the Hypothesis now are
\[
E(\vec{\f x})\xrightarrow{(\f p,j,\f q)}F^{1,1}(\vec{\f x})
\text{ and }E(\vec{\f x})\xrightarrow{(\f p,j,\f q')} F^{2,1}(\vec{\f x})
\]
where $\f q$ and $\f q'$ are distinct but have the same  sort and arity; so
\[
F^{1,1}\{\f q:\equiv \f r\}\equiv_c F^{2,1}\{\f q':\equiv \f r\}
\]
with any fresh $\f r$; and then  by the Induction Hypothesis and (3) of Lemma~\ref{redfree} as above,
\[
F^{1,k}\equiv_c F^{1,k}\{\f q:\equiv \f r\}\equiv_c F^{2,k}\{\f q':\equiv \f r\}\equiv_c F^2,
\]
which is what we needed to prove.
\end{proofplain}

\begin{theorem}[Canonical forms]
\label{cfthm}

Every extended program $E(\vec{\f x})$ is reducible to a unique up to congruence irreducible
extended program $\nf(E(\vec{\f x}))$, its canonical form.\smallskip

In detail: with every extended program $E(\vec{\f x})$, we can associate an
extended  program $\nf(E(\vec{\f x}))$ with the following properties:
\begin{ritemize}{(3)}

\item[(1)] $\nf(E(\vec{\f x}))$ is irreducible.

\item[(2)] $E(\vec{\f x})\redto\nf(E(\vec{\f x}))$.

\item[(3)] If $F(\vec{\f x})$ is irreducible and $E(\vec{\f x})\redto F(\vec{\f x})$,
then $F(\vec{\f x})\equiv_c\nf(E(\vec{\f x}))$.
\end{ritemize}

It follows that for all $E(\vec{\f x}),F(\vec{\f x})$,
\[
E(\vec{\f x})\redto F(\vec{\f x}) \implies \nf(E(\vec{\f x}))\equiv_c\nf(F(\vec{\f x})).
\]
\end{theorem}

\begin{proof}

If $E(\vec{\f x})$ is irreducible, take $\nf(E(\vec{\f x}))\equiv E(\vec{\f x})$, and
if $\size(E)=k>0$, fix a maximal sequence of one-step reductions
\[
E(\vec{\f x})\redto_1 F^1(\vec{\f x})\redto_1 \cdots \redto_1 F^k(\vec{\f x})
\]
and set $\nf(E(\vec{\f x})):\equiv F^k(\vec{\f x})$; now (1) and (2) are immediate and (3) follows
from Lemma~\ref{cfthmlemma}. The last claim holds because
$F(\vec{\f x})\redto\nf(F(\vec{\f x}))$, so
\[
\textup{if }E(\vec{\f x})\redto F(\vec{\f x}) \textup{ then } E(\vec{\f x})\redto\nf(F(\vec{\f x}))
\]
and hence $\nf(E(\vec{\f x}))\equiv_c\nf(F(\vec{\f x}))$ by (3).
\end{proof}

It is clearly possible to assign to each extended program $E(\vec{\f x})$ a specific
canonical form, \eg by choosing $\nf(E(\vec{\f x}))$ to be the
``lexicographically least'' irreducible $F(\vec{\f x})$ such that
$E(\vec{\f x})\redto F(\vec{\f x})$. This, however, is neither convenient nor useful, and
it is best to think of $\nf(E(\vec{\f x}))$ as denoting ambiguously
\textit{any irreducible extended program} such that $E(\vec{\f x})\redto\nf(E(\vec{\f x}))$; the
theorem insures that any two such \textit{canonical forms} of $E(\vec{\f x})$
are congruent, and the construction in the proof gives a simple
way to compute one of them.

\ysection{Canonical forms for richer languages}
A substantially more general version of the Canonical Form Theorem
for \textit{functional structures} was established
in~\cite{ynmflr}, and there are natural versions of it for many
richer languages, including a suitable formulation of the typed
$\lambda$-calculus with recursion (PCF). The proof we gave here for McCarthy
programs is considerably easier than these, more general results.

\ysection{Intensions and elementary algorithms}
\label{intdefinition}

The (referential)
\textit{intension} of an extended McCarthy program $E(\vec{\f x})$
in a structure $\A$ is the recursor of its canonical form,
\begin{equation}
\label{refintdef}
\refint(\A,E(\vec{\f x})) \edf \recursor(\A,\nf(E)(\vec{\f x}));
\end{equation}
it does not depend on the choice of $\nf(E)$
by~\eqref{congintequal}, and it models the \textit{elementary
algorithm} expressed by $E(\vec{\f x})$ in $\A$.

\ynm{A (bad) discussion of intensional equivalence between programs, which
I should at some point address.

An operation $\boldsymbol\phi(\A,E(\vec{\f x}))$ on the extended programs of a
structure $\A$ is an \textit{intensional invariant} if it depends
only on the intension of $E(\vec{\f x})$ in $\A$, \ie
\begin{equation}
\label{intinvariant}
\refint(\A,E(\vec{\f x})) = \refint(\A,F(\vec{\f x}))
\implies \boldsymbol\phi(\A,E(\vec{\f x})) = \boldsymbol\phi(\A,F(\vec{\f x})).
\end{equation}
We would expect that the most basic complexity measures on
programs are intensional invariants, and they are,
Problems~\ref{callsinvariant}, \ref{depthinvariant}.}

\ynm{In the structure $(A,\phi,\psi)$ where $\phi=\psi$ the terms
$\phi(x)$ and $\psi(x)$ have the same intension, but do not have
the same calls-complexity functions for $\phi$ because of the
names; also, if $\phi(x)=x$, then the total calls-complexity is
not determined by the intension. Most likely I should avoid these
picky issues here, it is not the place for them.}

\ysection{The recursor representation of an iterative algorithm}

Using the precise definition of the classical notion of an \textit{iterator} (sequential
machine)
\[
\iterator = (\yinput, S, \sigma, T, \youtput) : X\iterto W
\]
in Section~\ref{cmodels} of \ar, let
\[
\A_{\iterator} = (A_{\iterator}, X, W, S,
\yinput, \sigma, T, \youtput)
\]
be the structure associated with $\iterator$, where
\[
A_{\iterator} = X\uplus W\uplus S
\]
is the disjoint sum of the sets $X,W$ and $S$,  and let
\[
E_{\iterator}(\f x) :\equiv
\f q(\yinput(\f x))
\where
\Lbrace\f q(\f s) = \tif T(\f s)~\tthen \youtput(\f s)
~\telse \f q(\sigma(\f s))\Rbrace.
\]

\begin{proposition}
\label{recit}

An iterator $\iterator$ is determined by its intension
\begin{equation}
\label{recit2}
\refint(\iterator) = \refint(\A_\iterator, E_\iterator(\f x))
= \recursor(\A_\iterator,\nf(E_\iterator)(\f x)),
\end{equation}
the recursive algorithm on $\A_\iterator$ expressed by the extended tail recursive program
$E_\iterator(\vec{\f x})$ associated with $\iterator$.
\end{proposition}

\begin{proofplain} is easy, if a bit technical, and we leave it for
Problem~\ref{recitprob}.
\end{proofplain}

Theorem 4.2 in~\cite{ynmpaschalis2008} is a much stronger result
along these lines, but it requires several definitions for its
formulation and this simple Proposition expresses quite clearly
the basic message: every property of an iterator $\iterator$ can
be expressed as a property of its associated recursor
$\refint(\iterator)$, and so choosing recursive rather than
iterative algorithms as the basic objects does not force us to
miss any important facts about the simpler objects.

\ynm{\ysection{About implementations (3)}
\label{implementations3}

An outline of a \textit{theory of implementations} was sketched
first in~\cite{ynmsicily} and then (in more detail)
in~\cite{ynmpaschalis2008}. Briefly:\smallskip

(1) A \textit{reducibility relation} $\alpha\leq_r\beta$ on
recursors $\alpha,\beta:X\iterto W$ with the same input and output
sets is defined, aiming to capture an exact notion of
``simulation''.\smallskip

(2) By definition, an \textit{implementation} of a  recursor $\alpha:X\iterto W$
is any iterator $\iterator : X\iterto W$
such that $\alpha$ is
reducible to the recursor associated with $\iterator$,
$\alpha\leq_r\recursor(\A_\iterator,E_\iterator(\f x))$.\smallskip

By the main Theorem~5.3 in~\cite{ynmpaschalis2008},
\begin{quote}
\it
the recursive machine $\iterator(\A,E(\vec{\f x}))$ associated
with $E(\vec{\f x})$ in $\A$ is an implementation of the recursor
$\recursor(\A,E(\vec{\f x}))$.
\end{quote}
This provides some, initial evidence that this approach to a
theory of implementations at least covers a most basic case, but
not much more has been done on the question and we will not go
into it here. What is worth emphasizing (again) is that \textit{the
recursive machine $\iterator(\A,E(\vec{\f x}))$ lives in a
structure $\A_\iterator$ whose universe $A_\iterator$ extends the
universe $A$ of $\A$ and which has richer primitives}. In other
words, the usual assertion that \textit{recursion can be reduced
to} (or \textit{implemented by}) \textit{iteration} involves
extending the universe of the structure and adding new primitives;
and this is necessary by the results discussed in Section~\ref{recvstail}
and on page~\pageref{recvstailII} of \ar.}
\ynm{in~\cite{stolboushkin1983}
and~\cite{tiuryn1989} cited on page~\pageref{recistail}.}

\problems

\begin{prob}
\label{immsorts}

A term $I$ is \textit{immediate of sort $s$} if $I\equiv \f v_i$
and $s=\ind$; or $I\equiv \f p^{s,0}$; or $I\equiv \f p^{s,n}(\f
u_1,\ldots, \f u_n)$ where each $\f u_i$ is an individual variable.
Write out a similar, full (recursive) definition of \textit{irreducible terms
of sort $s$}.
\end{prob}

\ynm{\begin{prob}
\label{redfreeprob}
Prove Lemma~\ref{redfree}
\end{prob}}

\begin{prob}[Size]
\label{size0}

(1) Prove than there is exactly one function $\size(E)$ which assigns a number to
every explicit term so that
\begin{multline*}
\textit{if $E(\vec{\f x})$ is irreducible, then $\size(E)=0$, and}\\
\size(\f c(G_1,\ldots,G_l)) = \textstyle\sum\big\{\size(G_i)+1 \st G_i
\text{ is not immediate}\}.
\end{multline*}

(2) For a program $E$ as in~\eqref{recprog}, set
\[
\size(E) = \textstyle\sum\big\{\size(E)_i \st E_i \text{ is a part of }E\big\}
\]
and check that
\begin{enumerate}
\item[(1)] $E(\vec{\f x})$ is irreducible if and only if $\size(E)=0$, and
\item[(2)] if $E(\vec{\f x})\redto_1 F(\vec{\f x})$,  then $\size(F)=\size(E)-1$.
\end{enumerate}
\ynm{This was with only the intuitive def. around. \newcommand\tsum{\sum}
(1) If $E$ is irreducible, then every equation of $E$ is of the
form
\[
\f p_i(\vec{\f x}_i) = \f c(G_1,\ldots,G_l)
\]
with immediate $G_1,\ldots, G_l$ and no non-immediate term occurs
as a proper subterm of a part of $E$ and $\size(E)=0$; and
if $E$ is reducible, then in some equation of $E$ as above some
$G_i$ is not immediate and is counted, so
$\size(E)>0$.

(2) If $\f p_i(\vec{\f x}_i) = \f
c(G_1,\ldots,G_{j-1},G_j,G_{j+1},\ldots,G_l)$ is an equation of
$E$, then the contribution it makes to the size of $E$ is
\[
s = \tsum_{1\leq i \leq l}\size(G_i) + |\{i \st G_i \text{ is not immediate}\}|.
\]

If $G_j$ is not immediate, then the two equations in the program $E'$
constructed by the one-step reduction contribute
\begin{multline*}
s' = \tsum_{1\leq i\leq l, i\neq j}\size(G_i)\\
+ |\{i \st i\neq j \conj G_i \text{ is not immediate}\}|
+\size(G_j) = s-1
\end{multline*}
to the $\size(E')$.}
\end{prob}

\begin{prob}
\label{noopt}

Notice that the irreducible extended program $F_5(\vec{\f x})$ in
Diagram~\ref{redexample} to which $E(\vec{\f x})$ reduces calls for computing
the value $\phi_1(x)$ twice in each ``loop''. Define an extended program $F(\vec{\f x})$
which computes the same partial function as $E(\vec{\f x})$ but calls for
computing $\phi_1(x)$ only once in each loop, and construct a
complete reduction sequence for $E^\ast(\vec{\f x})$.

\sol{$E^\ast \equiv \tif\phi_0(q(x)~\tthen y~\telse \phi_2(q(x),y)
\where\{q(x) = \phi_1(x)\}$.}
\end{prob}

\begin{dprob}
\label{recitprob}

Prove Proposition~\ref{recit}. \hint Use the reduction calculus to
compute a canonical form of the program $E_\iterator(\vec{\f x})$,
\begin{multline*}
\nf(E_{\iterator}(\vec{\f x}))\equiv_c \f q(\f p_1(\f x)) \rwhere \Lbrace \f p_1(\f x) = \yinput(\f x),
\\
\f p_2(\f s) = T(\f s),~~\f p_3(\f s) = \youtput(\f s),~~
\f p_4(\f s) = \sigma(\f s),~~\f p_5(\f s) = \f q(\f p_4(\f s))\\
\f q(\f s) = \tif \f p_2(\f s)~\tthen \f p_3(\f s)~\telse \f p_5(\f s)\Rbrace,
\end{multline*}
so that $\refint(\iterator(\f x))=\refint(\A_\iterator,E_\iterator(\f
x)) =\recursor(\A_\iterator,\nf(E_\iterator(\f x)))=
(\alpha_0,\alpha_1,\ldots,\alpha_6)$, where
\[
\alpha_i: A_\iterator\times (A_\iterator\pfto A_\iterator)
\times(A_\iterator\pfto\boolset)\times(A_\iterator\pfto
A_\iterator)^4 \pfto A_\iterator\quad(i=1,\ldots,6).
\]
The result is practically trivial from this, except for the fact
that we only know the canonical form up to congruence, so it is
not immediate how to ``pick out'' from the functionals $\alpha_i$
the input, output, \etc of $\iterator$---and this is complicated
by the fact that the sets $X,W$ and $S$ may overlap, in fact they may all
be the same set. We need to use the details of the coding of
many-sorted structures by single-sorted ones specified in Section~1D of~\ar.

\sol{The functionals $\alpha_i$ determine the set
\[
A_{\iterator} = X\uplus W\uplus S = (\{0\}\times X)\cup (\{1\}\times W) \cup
(\{2\}\times S)
\]
and hence the sets $X,W$ and $S$ of the iterator.
Moreover, exactly four of them are independent of their partial
function arguments, \ie (with the listing we assumed)
\[
\alpha_1(x,\vec p) = f_1(x),~~\alpha_2(s,\vec p)=f_2(s), ~~
\alpha_3(s,\vec p)=f_3(s),~~\alpha_4(s)=f_4(s)
\]
for suitable partial functions $f_1,f_2,f_3,f_4$; and then we can
easily read off them the set $T$ of terminal states and the
partial functions
\[
\yinput:X\to S, \youtput:T\to W, \sigma:S\pfto S.
\]

For example, $\yinput$ is the only one of these functions whose
input set $\{0\}\times X$;
$f_2$ (with our listing) is the only one of these functions
whose output set is $\boolset$ and it determines $T$; $f_3$ is
the only one of these functions whose output set is $\{1\}\times
W$ and it determines $\youtput$; and then $\sigma$ is the only one left.
}
\end{dprob}

\section{Decidability of intensional program equivalence}
\label{decidability}\ynm{Why did I put this note here?
The reduction to irreducible identities needs to be checked.}

Two extended $\Phi$-programs are \textit{intensionally equivalent
on a $\Phi$-structure} $\A$ if they have equal intensions,
\begin{equation}
\label{inteq}
E(\vec{\f x})\inteq{\A} F(\vec{\f x}) \eedf
\refint(\A,E(\vec{\f x})) = \refint(\A,F(\vec{\f x}));
\end{equation}
they are (globally) \textit{intensionally equivalent} if they are
intensionally equivalent on every infinite $\A$,
\begin{equation}
\label{ginteq}
E(\vec{\f x})\inteq{} F(\vec{\f x}) \eedf
(\forall\text{ infinite }\A)[\refint(\A,E(\vec{\f x})) = \refint(\A,F(\vec{\f x}))].
\end{equation}

The main result of this section is

\begin{theorem}
\label{localdec}
For every infinite $\Phi$-structure $\A$, the relation of
intensional equivalence on $\A$ between extended $\Phi$-programs
is decidable.\footnote{Except for a reference to a published
result in Problem~\ref{intequivlowerbound}, I will not discuss in
this paper the question of \textit{complexity} of intensional
equivalence on recursive programs, primarily because I do not know
anything non-trivial about it.}
\end{theorem}

We will also prove that global intensional equivalence between
$\Phi$-programs is essentially equivalent with congruence, and so
it is decidable, Theorem~\ref{globaldec}.

\ysection{Plan for the proof}
\label{proofplan}
To decide if $E(\vec{\f x})$ and $F(\vec{\f x})$ are intensionally
equivalent on $\A$, we compute their canonical forms, check
that these have the same number $K+1$ of irreducible parts $E_i, F_j$
and then (to apply Lemma~\ref{recprogiso}) check if there is a
permutation $\pi$ of $\{0,\ldots,K\}$ with inverse $\rho$ such
that $\pi(0)=0$ and
\[
\begin{array}{rcll}
\A &\models& E_0(\vec{\f x})&=
F_0\{\vec{\f q} :\equiv \rho(\vec{\f p})\}(\vec{\f x}),\\
\A &\models& E_i(\vec{\f x}_i)&= F_{\pi(i)}\{\vec{\f q} :\equiv \rho(\vec{\f
p}), \vec{\f y_i}_{\pi(i)}:\equiv\vec{\f x}_i\}, \quad (i=1,\ldots,K).
\end{array}
\]
These identities are all irreducible by~\eqref{alphchange}, so the
problem comes down to deciding the validity of \textit{irreducible
term identities} in $\A$. The proof involves associating with each
infinite $\Phi$-structure $\A$ a finite list of ``conditions''
about its (finitely many) primitives---its
\textit{dictionary}---which codifies all the information needed to
decide whether an arbitrary irreducible identity holds in $\A$; it
involves some fussy details, but fundamentally it is an elementary
exercise in the \textit{equational logic of partial terms with
conditionals}. In classical terminology, it is (basically) a
\textit{Finite Basis Theorem} for the theory of irreducible
identities which hold in an arbitrary, infinite (partial)
$\Phi$-structure.\footnote{The decidability of intensional equivalence is
considerably simpler for total structures $\A$, \cf
Lemma~\ref{totalirreducible} and substantially more difficult for
the \textsf{FLR}-structures of~\cite{ynmflr} which allow
\textit{monotone, discontinuous functionals} among their
primitives; see \cite{ynmfrege}---and a correction which fills a
gap in the proof of this result in~\cite{ynmfrege} which is posted
(along with the paper) on ynm's homepage. It is an open---and, I
think, interesting---question for the \textit{acyclic recursive
algorithms} of~\cite{ynmlcms}, where intensional synonymy  models \textit{synonymy}
(or \textit{faithful translation}) for fragments of natural
language.}

\ysection{The satisfaction relation}

We will also need in this Section the classical characterization
of valid identities in terms of the \textit{satisfaction
relation}.

An \textit{assignment} $\sigma$ in a $\Phi$-structure $\A$ associates with each
variable an object of the proper kind, \ie
\[
\sigma(\f v_i)\in A,\quad \sigma(\f p^{s,n}_i):A^n\pfto A_s,
\]
so that, for example, the \textit{boolean function variables} $\f
p^{\bool,0}_i$ are assigned nullary partial functions
$p:\emptyproduct\pfto \boolset$, essentially $\diverges$,
$\true$ or $\false$. For each $\tPhi$-term $E$, we define
\[
\sigma(\A,E) \edf\text{ the value (perhaps $\diverges$)
of $E$ in $\A$ under $\sigma$}
\]
by the usual compositional recursion and we set
\[
\A,\sigma \models E=F \eedf \sigma(\A,E) = \sigma(\A,F).
\]
If the variables of $E$ are among $\f p_1, \ldots, \f p_K, \f x_1, \ldots, \f x_n$, then
\[
\sigma(\A,E) = \den((\A,\sigma(\f p_1), \ldots, \sigma(\f p_K)),
\sigma(\f x_1),\ldots, \sigma(\f x_n)),
\]
and so by the notation on p.~\pageref{models}, for any two extended $\tPhi$-terms,
\[
\A\models E(\vec{\f x}) = F(\vec{\f x}) \iff (\text{for all } \sigma)[\A,\sigma\models E=F].
\]

\ysection{Preliminary constructions and notation}

We fix for this Section an infinite  $\Phi$-structure
$\A=(A,\Phi)$ (with finite $\Phi$, as always) and we assume
without loss of generality that among the symbols in $\Phi$ are
$\psi_\true, \psi_\false, \id$ such that
\[
\psi_\true^\A=\true, \quad \psi_\false^\A = \false, \quad \id^\A(t) = t.
\]

We will often define assignments partially, only on the variables
which occur in a term $E$ and specifying only finitely many values
$\sigma(\f p^{s,n}_i)(\vec x)$ of the function variables---those needed
to compute $\sigma(\A,E)$---perhaps
to come back later and extend $\sigma$ when we view $E$ as a
subterm of some $F$.\smallskip

Using the assumption that $A$ is infinite, we fix pairwise disjoint, infinite sets
\[
A^\ind, ~~A^{\ind,n}_i \subset A
\]
whose union is co-infinite in $A$ and we fix an assignment $\tsigma$ on the individual
and function variables of sort $\ind$ using injections
\[
\tsigma:\{\f v_0, \f v_1, \ldots\} \inj A^\ind, \quad \tsigma(\f p^{\ind,n}_i):A^n\inj A^{\ind,n}_i.
\]
We do not define at this point any values $\tsigma(\f
p^{\bool,n}_i)(x_1,\ldots,x_n)$ for function variables of sort
$\bool$.\smallskip

A term $E$ is \textbf{pure, algebraic} if it is of
sort $\ind$ and none of $\true, \false$, the conditional or any
symbol from $\Phi$ occurs in it, \ie
\begin{equation}
\tag{Pure, algebraic terms}
E :\equiv \f v_i \mid \f p^{\ind,0}_i \mid \f p^{\ind,n}_i(E_1,\ldots,E_n).
\end{equation}

\begin{lemma}
\label{purealgebraic}

If $E,F$ are pure, algebraic terms, then
\[
\A,\tsigma \models E =F \iff E\equiv F.
\]
\end{lemma}

\begin{proofplain} of this (classical) result is by an easy induction on $\ylength(E)$, using
the Parsing Lemma in Problem~\ref{fulltermdef}.
\end{proofplain}

\ysection{Forms of irreducible identities}

To organize the proof of Theorem~\ref{localdec}, we will appeal to the (trivial) fact that
\textit{every irreducible term is in exactly one of the following
forms}:
\begin{enumerate}
\label{irrterms}

\item[(1)] $\true$, $\false$, or an individual variable $\f v$.

\item[(2)] Function variable application, $\f p(z_1,\ldots,z_n)$,
where the terms $z_1,\ldots,z_n$ are immediate of sort $\ind$ and
$n\geq0$; $\f p$ can be of either sort, $\ind$ or $\bool$.

\item[(3)] Conditional, $\tif z_1~\tthen z_2~\telse z_3$, with
immediate $z_1$ of sort $\bool$ and immediate $z_2,z_3$ of the
same sort, $\ind$ or $\bool$.

\item[(4)] Primitive application, $\phi(z_1,\ldots,z_k)$, $k\geq
0$, with immediate $z_1,\ldots,z_k$, all of sort $\ind$ and $\phi$
of either sort.
\end{enumerate}

It follows that every irreducible identity falls
in one of the ten, pairwise exclusive forms \ef{i}{j}  with $1\leq
i\leq j\leq 4$ depending on the forms of its two sides.
In dealing with these cases in the following Lemmas, we  will use
repeatedly the (trivial) fact that
\begin{quote}
\it if $u,v$ are immediate of sort $\bool$
and $u\not\equiv v$, then we can extend $\tsigma$ to the function
variables which occur in $u$ and $v$ so that $\tsigma(u)$ and
$\tsigma(v)$ take any pre-assigned truth values \tu{or diverge}.
\end{quote}

\begin{lemma}[\ef{2}{2}]
\label{2-2}
For any two irreducible terms in form \tu{2},
\[
\A\models\f p(z_1,\ldots,z_n) = \f q(w_1,\ldots,w_m) \iff \f p(z_1,\ldots,z_n) \equiv \f q(w_1,\ldots,w_m).
\]
\end{lemma}

\begin{proof}
Let $E\equiv \f p(z_1,\ldots,z_n)$
and $F\equiv \f q(w_1,\ldots,w_m)$, assume $\A\models E=F$ and consider two cases:

\textit{Case 1}, $\sort(\f p)=\sort(\f q)=\ind$; now $E$ and $F$ are pure algebraic
and so by Lemma~\ref{purealgebraic}:
\[
\A\models E=F\implies \A,\tsigma\models E=F \implies E\equiv F.
\]

\textit{Case 2}, $\sort(\f p)=\sort(\f q)=\bool$. If $\f p\not\equiv \f q$, extend $\tsigma$ by setting
\[
\tsigma(\f p)(x_1,\ldots,x_n)=\true, \quad \tsigma(\f q)(y_1,\ldots,y_m)=\false
\]
and check that $\A,\tsigma\not\models E=F$;
so $\f p\equiv \f q$, hence $n=m$,
and it is enough to prove that $z_1\equiv w_1, \ldots, z_n\equiv w_n$. If not,
then for some $i$, $\tsigma(z_i)\neq\tsigma(w_i)$ by Lemma~\ref{purealgebraic}; and then we can choose some
$a\in A$ and extend $\tsigma$ by setting
\[
\tsigma(\f p)(x_1,\ldots,x_n) =\tif (x_i=a)~\tthen\true~\telse\false,
\]
so that $\tsigma(\f p(z_1,\ldots,z_n))\neq \tsigma(\f p(w_1,\ldots,w_n))$, which contradicts
the hypothesis.
\end{proof}

\begin{lemma}[\ef{3}{3}]
\label{3-3}
If $z_1, \ldots, w_3$ are all immediate, then
\begin{multline*}
\A\models \tif z_1~\tthen z_2~\telse z_3 = \tif w_1~\tthen w_2~\telse w_3,
\\
\iff z_1\equiv w_1, z_2\equiv w_2, z_3\equiv w_3.
\end{multline*}
\end{lemma}

\begin{proof}
Assume $\A\models \tif z_1~\tthen z_2~\telse z_3=\tif
w_1~\tthen w_2~\telse w_3$, note that $z_1,w_1$ must be of sort
$\bool$ and consider cases:\smallskip

\textit{Case} (a): $z_2,z_3,w_2,w_3$ are of sort $\ind$. If $z_1\not\equiv w_1$,
extend $\tsigma$ so that
\[
\tsigma(z_1)=\true, \quad \tsigma(w_1)\diverges
\]
which gives $\tsigma(\tif z_1~\tthen z_2~\telse z_3)=\tsigma(z_2) \text{
while } \tsigma(\tif w_1~\tthen w_2~\telse w_3)\diverges$,  hence
$z_1\equiv w_1$. Setting $\tsigma(z_1)=\tsigma(w_1)=\true$ gives $z_2\equiv w_2$ and
setting $\tsigma(z_1)=\false$ gives $z_3\equiv w_3$.\smallskip

\textit{Case} (b): all $z_1,z_2,z_3,w_1,w_2,w_3$ are of sort $\bool$.  Assume first,
towards a contradiction that $z_1\not\equiv w_1$, and set
\[
\tsigma(z_1)=\true, \quad \tsigma(w_1)\diverges.
\]
If $z_2\not\equiv w_1$, we can further set $\tsigma(z_2)=\true$ (even if $z_2\equiv z_1$) and
get a contradiction, so $z_2\equiv w_1$. The same argument, with $\tsigma(z_1)=\false$ this time
gives $z_3\equiv w_1$; and the symmetric argument gives $w_2\equiv w_3\equiv z_1$,
so what we have now is
\[
\A\models \tif z_1~\tthen w_1~\telse w_1 = \tif w_1~\tthen z_1~\telse z_1
\]
with the added hypothesis that $z_1\not\equiv w_1$; but then we can set
\[
\tsigma(z_1)=\true, \quad \tsigma(w_1)=\false,
\]
which gives
\[
\tsigma(\tif z_1~\tthen w_1~\telse w_1) = \tsigma(w_1)=\false,
\quad
\tsigma(\tif w_1~\tthen z_1~\telse z_1)=\tsigma(z_1)=\true,
\]
which violates the hypothesis; so $z_1\equiv w_1$, and the hypothesis in Case (b) takes the
form
\[
\A\models \tif z_1~\tthen z_2~\telse z_3 = \tif z_1~\tthen w_2~\telse w_3.
\]

Now assume towards a contradiction that $z_2\not\equiv w_2$; so one of them is different
from $z_1$, suppose it is $w_2$; so we now get a contradiction by setting
$\tsigma(z_1)=\tsigma(z_2)=\true$ and $\tsigma(w_2)=\false$.

At this point we know that $z_1\equiv w_1$ and $z_2\equiv w_2$, and the last bit that
$z_3\equiv w_3$ is simpler.
\end{proof}

\begin{lemma}[\ef{2}{3}]
\label{2-3}
If $z_1,\ldots,z_m, w_1, w_2, w_3$ are all immediate, then
\begin{multline*}
\A\models
\f p(z_1,\ldots,z_m)=\tif w_1~\tthen w_2~\telse w_3
\\
\iff \f p(z_1,\ldots,z_m)\equiv w_1\equiv w_2\equiv w_3.
\end{multline*}
\end{lemma}

\begin{proof}
Notice first that if $\f p$ is of sort $\ind$, then
\[
\A,\sigma \not\models \f p(z_1,\ldots,z_m)=\tif w_1~\tthen w_2~\telse w_3
\]
for any $\sigma$ which converges on $\f p(z_1,\ldots,z_m)$
but such that $\sigma(w_1)\diverges$.\smallskip

If $\f p$ is of sort $\bool$ and $\f p(z_1,\ldots,z_m)\not\equiv w_1$,
then we can get a $\sigma$ such that
$\sigma(\f p(z_1,\ldots,z_m))\converges$ but $\sigma(w_1)\diverges$
so the hypothesis
fails again. It follows that $\f p(z_1,\ldots,z_m)\equiv w_1$ and the
hypothesis takes the form
\[
\A\models w_1=\tif w_1~\tthen w_2~\telse w_3.
\]
If $w_2\not\equiv w_1$, we get a contradiction by setting
$\sigma(w_1)=\true, \sigma(w_2)\diverges$,
and similarly if $w_1\not\equiv w_3$.
\end{proof}

\begin{lemma}[\ef{3}{4}]
\label{3-4}
If $z_1,\ldots,z_m, w_1, w_2, w_3$ are all immediate, then
\[
\A\not\models
\phi(z_1,\ldots,z_m)=\tif w_1~\tthen w_2~\telse w_3.
\]
\end{lemma}

\begin{proof}
Notice that $w_1\not\equiv z_i$, since $\sort(w_1)=\bool$
while $\sort(z_i)=\ind$ and consider three cases:\smallskip

If $\phi^\A(\tsigma(z_1),\ldots,\tsigma(z_m))\converges$, extend
$\tsigma$ by setting $\tsigma(w_1)\diverges$ so
\begin{equation}
\tag{$\ast$}
\label{star34}
\A,\tsigma\not\models
\phi(z_1,\ldots,z_m)=\tif w_1~\tthen w_2~\telse w_3.
\end{equation}

If $\phi^\A(\tsigma(z_1),\ldots,\tsigma(z_m))\diverges$ and $\phi$ is of sort $\ind$,
then setting now $\tsigma(w_1)=\true$ gives~\eqref{star34} again,
since $\tsigma(w_2)\converges$.\smallskip

If $\phi^\A(\tsigma(z_1),\ldots,\tsigma(z_m))\diverges$ and $\phi$ is of sort $\bool$,
then we can set $\tsigma(w_1)=\tsigma(w_2)=\true$ which again gives~\eqref{star34}.
\end{proof}

\begin{lemma}
\label{philemma1}
If $z_1, \ldots, z_k, w_1, \ldots,w_m$ are all immediate, then
\begin{multline*}
\text{if }\A\models \phi(z_1,\ldots,z_k) = \f p(w_1,\ldots,w_m),\\
\text{then $\f p(w_1,\ldots,w_m)\equiv z_j$ for some $j$}\\
\text{and so }\A\models \phi(z_1,\ldots,z_k)=z_j=\id(z_j).
\end{multline*}
In particular, with $k>0$ and $m=0$,
\[
\A\models \phi(z_1,\ldots,z_k) = \f p \implies
\f p\equiv z_j \text{ for some }j,
\]
and with $k=0$, $\A\not\models \phi=\f p(z_1,\ldots,z_m)$.
\end{lemma}

\begin{proof}
Assume the hypothesis $\A\models \phi(z_1,\ldots,z_k) =
\f p(w_1,\ldots,w_m)$ and also that $\f p(w_1,\ldots,w_m)\not\equiv z_j$ for any $j$.\smallskip

(a) $\sort(\phi)=\ind$; because if $\sort(\phi)=\bool$, then
\begin{multline*}
\text{either }\phi^\A(\tsigma(z_1),\ldots,\tsigma(z_k)))\converges
\text{ and we can set }\tsigma(\f p)(\tsigma(z_1),\ldots,\tsigma(z_k))\diverges\\
\text{or } \phi^\A(\tsigma(z_1),\ldots,\tsigma(z_k)))\diverges
\text{ and we can set }\tsigma(\f p)(\tsigma(z_1),\ldots,\tsigma(z_k))\converges.
\end{multline*}

($\ast$) Call $E$ \textit{relevant} if $E\equiv z_i$ or $E\equiv w_i$ for
some $i$, set $\sigma(\f x) = \tsigma(\f x)$ for every individual
variable which occurs in a relevant term, and for a function variable $\f q^{\ind,n}$ of sort
$\ind$ and arity $n$, set
\begin{multline*}
\sigma(\f q^{\ind,n})(x_1,\ldots,x_n) = v \eedf \text{ there is a relevant }F
   \equiv \f q^{\ind,n}(\f x_1,\ldots,\f x_n)\\
\text{such that }x_1=\tsigma(\f x_1), \ldots, x_n=\tsigma(\f x_n),
\conj\tsigma(\f q^{\ind,n})(x_1,\ldots,x_n)=v.
\end{multline*}

(b) \textit{For each pure, algebraic term $E$, $\sigma(E)$ is
defined exactly when $E$ is relevant and then
}$\sigma(E)=\tsigma(E)$. This is because by the definition, if
$\sigma(E)$ is defined, then $\sigma(E)=\tsigma(E)=\tsigma(F)$ for
a relevant $F$ and hence $E\equiv F$ by
Lemma~\ref{purealgebraic}.\smallskip

It follows that $\sigma(\f p(w_1,\ldots,w_m))$ is not defined by the stipulation
($\ast$), since
$\f p(w_1,\ldots,w_m)\not\equiv z_i$ by the hypothesis and each $w_j$
is a proper subterm of  $\f p(w_1,\ldots,w_m)$, so it cannot be equal
to it; and then we reach a contradiction by setting $\sigma(\f p)$ so that
\begin{endproofeqnarray*}
\sigma(\f p)(\tsigma(w_1),\ldots,\tsigma(w_m))\neq\phi^\A(\tsigma(z_1),\ldots,\tsigma(z_k).
\end{endproofeqnarray*}
\end{proof}

\begin{lemma}
\label{redto44}

If there is an algorithm which decides the validity in $\A$ of identities in form
\ef{4}{4}.
\[
\A\models \phi(z_1,\ldots,z_n)=\psi(z_{n+1},\ldots,z_l)\quad
(\phi,\psi\in\Phi,z_1,\ldots,z_l \textup{ immediate}),
\]
then the relation of intensional equivalence on $\A$ between
extended $\Phi$-programs is decidable.

\end{lemma}

\begin{proof}

\ynm{It is useful to assume here that there are nullary constants $\phi_\true,
\phi_\false$ of sort $\bool$ and a unary function symbol $\id$ of
sort $\id$ in $\Phi$ such that
\begin{equation}
\label{tfid}
\phi_\true^\A = \true,\quad \phi_\false^\A=\false \text{~~ and ~~ }\id^\A(x) = x,
\end{equation}
by adding them to $\Phi$ if necessary.\smallskip}

Forms \ef{1}{1}, \ef{1}{2} and \ef{1}{3} are trivial, \cf Problem~\ref{1-123}.\smallskip

Form \ef{2}{2} is decided by Lemma~\ref{2-2} and forms \ef{2}{3}, \ef{3}{3}
and \ef{3}{4} are decided by Lemmas~\ref{2-3}, \ref{3-3} and \ref{3-4}.\smallskip

Identities in forms \ef{4}{1} (equivalent to \ef{1}{4}) and \ef{4}{2}
(equivalent to \ef{2}{4}) are either
\[
\phi(z_1,\ldots,z_n)=\true=\psi_\true \text{ and }\phi(z_1,\ldots,z_n)=\false=\psi_\false
\]
which are in form \ef{4}{4} with our assumption that $\psi_\true, \psi_\false\in\Phi$,
or by Lemma~\ref{philemma1} and our assumption that $\id\in\Phi$,
they are equivalent in $\A$ to identities
\[
\phi(z_1,\ldots,z_n)=z_j=\id(z_j)
\]
for some $j$ (that we can compute), \cf Problem~\ref{14}.
\end{proof}

\textit{\bfseries At this point there is a fork in the road}: it is very easy to
finish the proof of Theorem~\ref{localdec} for total structures,
while the general case of arbitrary partial structures involves
some technical difficulties. We do the simple thing first.

\begin{lemma}
\label{totalirreducible}

Suppose $\A$ is a total, infinite, $\Phi$-structure,
$\phi, \psi\in\Phi$  and $z_1,\ldots,z_l$ are immediate terms.\smallskip

\tu{1} If $\A\models \phi(z_1,\ldots,z_n)=\psi(z_{n+1},\ldots,z_l)$,
then for all $i=1,\ldots,n$,
\[
\text{if $z_i$ is not an individual variable, then there
is a $j$ such that }z_i\equiv z_{n+j}.
\]

\tu{2} There is a sequence $\f x_1,\ldots,\f x_l$ of
\textup{(not necessarily distinct)} individual variables such that
\[
(\forall i,j)\Big(\f x_i\equiv \f x_j \iff z_i\equiv z_j\Big),
\]
and for any such sequence
\begin{multline}
\label{totalred}
\A\models \phi(z_1,\ldots,z_n)=\psi(z_{n+1},\ldots,z_l)\\
\iff
\A\models \phi(\f x_1,\ldots,\f x_n)=\psi(\f x_{n+1},\ldots,\f x_l).
\end{multline}

\end{lemma}

\begin{proof}
(1) Suppose towards a contradiction that
\[
\A\models \phi(z_1,\ldots,z_n)=\psi(z_{n+1},\ldots,z_l)
\]
but $z_i\equiv \f p(\f x_1,\ldots,\f x_k)$ is not identical with any
$z_{n+j}$ and define the assignment $\sigma$ which agrees with
$\tsigma$ on all individual and function variables except that
\[
\sigma(\f p)(\tsigma(\f x_1),\ldots,\tsigma(\f x_k))\diverges;
\]
this gives $\phi^\A(\sigma(z_1),\ldots,\sigma(z_n))\diverges$ and
$\phi^\A(\sigma_{n+1},\ldots\sigma(z_l))\converges$,
which contradicts the hypothesis.\smallskip

(2) To construct a sequence $\f x_1,\ldots, \f x_l$ of individual  variables with the
required property by induction on $i$:
pick a fresh $\f x_i$ if $z_i\not\equiv z_j$ for every  $j<i$ and otherwise let
$\f x_i :\equiv \f x_j \text{ for the least (and hence every) $j<i$ such that $z_i\equiv z_j$}$.

Suppose $\f x_1,\ldots,\f x_l$ are such that $\f x_i\equiv \f x_j \iff
z_i\equiv z_j$, assume that
\[
\A\models \phi(z_1,\ldots,z_n)=\psi(z_{n+1},\ldots,z_l)
\]
and for any assignment $\sigma$ define $\tau$ so that
\[
(\forall i=1,\ldots,k)[\tau(z_i) = \sigma(\f x_i)],
\]
which is possible by the assumption relating $\f x_i$ and $z_i$. The hypothesis gives
\[
\phi^\A(\tau(z_1),\ldots,\tau(z_n))=\psi^\A(\tau(z_{n+1}),\ldots,\tau(z_l)),
\]
which then implies
\[
\phi^\A(\sigma(\f x_1),\ldots,\sigma(\f x_n))=\psi^\A(\sigma(\f x_{n+1}),\ldots,\sigma(\f x_l)),
\]
and since $\sigma$ was arbitrary we get the required
\[
\A\models \phi(\f x_1,\ldots,\f x_n)=\psi(\f x_{n+1},\ldots,\f x_l).
\]
The converse is proved similarly.
\end{proof}

The point of the Lemma is that, for given $\phi,\psi$, there are
infinitely many identities which may or may not satisfy the
left-hand-side of~\eqref{totalred}, because there are infinitely many
immediate terms, \eg $z_1$ could be any one of
\[
\f p_1^1(\f v_1), \f p^2_1(\f v_1,\f v_1), \f p^3_1 (\f v_1, \f v_1, \f v_1), \ldots;
\]
while the right-hand-side of~\eqref{totalred} involves only
finitely many identities (up to alphabetic change),  which we can
exploit as follows:\smallskip

An \textbf{individual bare identity} in $\Phi$ is any identity in form \ef{4}{4}
\begin{equation}
\tag{$\ast$}
\theta : \phi(\f x_1,\ldots,\f x_k) = \psi(\f x_{k+1},\ldots,\f x_l),
\end{equation}
where the (not necessarily distinct) individual variables
$\f x_1,\ldots,\f x_l$ are chosen from the first $l$ entries in a fixed
list $\f v_0, \f v_1, \ldots$ of all individual variables; and the
\textbf{individual dictionary} of a total structure $\A$ is the
set
\begin{equation}
\tag{Ind-Dictionary}
\label{inddictionary}
D(\A) = \{\theta \st \theta \text{ is an individual bare identity and }\A\models \theta\}.
\end{equation}
This is a finite set (because $\Phi$ is finite) and by (2) of Lemma~\ref{totalirreducible},
for any identity in form \ef{4}{4} we can construct an individual bare identity such that
\begin{multline}
\label{inddicequiv}
\A\models \phi(z_1,\ldots,z_k)=\psi(z_{k+1},\ldots,z_l)\\
\iff \phi(\f x_1,\ldots,\f x_k)=\psi(\f x_{k+1},\ldots,\f x_l)\in D(\A);
\end{multline}
and then Lemma~\ref{redto44} then gives immediately

\begin{theorem}
\label{totaldecide}
For every total, infinite $\Phi$-structure $\A$, the relation of
intensional equivalence on $\A$ between extended $\Phi$-programs
is decidable.
\end{theorem}

\ysection{The general case} We now turn to the proof of
    Theorem~\ref{localdec} for an arbitrary infinite partial
$\Phi$-structure $\A$ which is similar in structure but requires
some additional arguments. It will be useful to keep in mind the
following example of an irreducible identity which illustrates
the changes that we will need to make to the argument:
\begin{equation}
\label{phipsiex}
\phi(\f p(\fs),\fs,\ft)=\psi(\f p(\ft),\fs,\ft).
\end{equation}

An individual variable $\fs$ is \textbf{placed} in an irreducible identity
\begin{equation}
\label{phipsieq}
\phi(z_1,\ldots,z_k) = \psi(z_{k+1}, \ldots,z_l)
\end{equation}
if $\fs\equiv z_i$ for some $i$; and an assignment $\sigma$ is
\textbf{injective} if it assigns distinct values
$\sigma(\fs)\neq\sigma(\ft)$ in $A$ to distinct, placed individual
variables $\fs\not\equiv\ft$. We write
\begin{multline}
\label{injective}
\A \models_\vinj \phi(z_1,\ldots,z_k) = \psi(z_{k+1}, \ldots,z_l)
\\ \eedf
\text{ for every injective }\sigma,~ \A,\sigma\models
\phi(z_1,\ldots,z_k) = \psi(z_{k+1}, \ldots,z_l).
\end{multline}

\begin{lemma}
\label{phipsilemma}

With every irreducible identity~\eqref{phipsieq}, we can associate a sequence
$\fx_1,\ldots,\fx_l$ of \textup{(not necessarily distinct)} individual
and nullary function variables of sort $\ind$ so that for every
\textup{(infinite)} $\Phi$-structure $\A$,
\begin{multline}
\label{phipsi}
\A\models_\vinj\phi(z_1,\ldots,z_k) = \psi(z_{k+1}, \ldots,z_l)
\\
\iff \A\models_\vinj \phi(\fx_1,\ldots,\fx_k) = \psi(\fx_{k+1}, \ldots,\fx_l).
\end{multline}
\end{lemma}

\begin{proof}

Call $i$ \textit{new} (in~\eqref{phipsieq}) if
there is no $j<i$ such that $z_i\equiv z_j$, and set
(by induction on $i$):
\begin{enumerate}
\item[(a1)] $\fx_i:\equiv z_i$, if $i$ is new and $z_i$
is an individual or nullary function variable;

\item[(a2)] $\fx_i :\equiv $  some fresh nullary function variable
(distinct from every $z_j$ and from every $\fx_j$
with $j<i$), if $i$ is new and $z_i\equiv \f p(\fs_1,\ldots,\fs_n)$ with
$n>0$;

\item[(a3)] $\fx_i:\equiv \fx_j$ for the least (and hence every) $j<i$ such that
$z_i\equiv z_j$, if $i$ is not new.

\end{enumerate}
Notice that directly from the definition,
\begin{equation}
\label{phipsi1}
\fx_i\equiv \fx_j \iff z_i\equiv z_j \quad(1\leq i,j\leq l),
\end{equation}
\cf Problem~\ref{phipsi1prob}.\smallskip

To prove first the direction $(\Longrightarrow)$ in~\eqref{phipsi}, assume that
\[
\A\models_\vinj\phi(z_1,\ldots,z_k) = \psi(z_{k+1}, \ldots,z_l),
\]
let $\tau$ be any injective assignment, and define $\sigma$ by setting first
\[
\sigma(z_i):=\tau(\fx_i), \text{ if $z_i$ is an individual or nullary function variable}.
\]
Since $\tau$ is injective, this already insures that, however we extend it,
$\sigma$ will be an injective assignment.

Next, if $\fs$ is an individual variable which occurs in some
$z_i\equiv \f p(\fs_1,\ldots,\fs_n)$ and is not placed, set
$\sigma(\fs)=\oo\fs$ to a fresh value in $A$, so that
\[
\fs\not\equiv \ft \implies \oo \fs\neq \oo \ft.
\]

At this point we have defined $\sigma$ on all the individual and
nullary function variables which occur in~\eqref{phipsi} and it
assigns distinct values to distinct individual variables. To
define it on $n$-ary function variables with $n>0$ which occur
in~\eqref{phipsi}, set
\[
\sigma(\f p)(\oo \fs_1,\ldots,\oo \fs_n) = \tau(\fx_i),
\text{ if }\f p(\fs_1,\ldots,\fs_n)\equiv z_i;
\]
which is a good definition, because if it happens that $\oo
\fs_1=\oo \ft_1,\ldots,\oo \fs_n=\oo \ft_n$ for some variables
$\ft_1,\ldots,\ft_n$ such that $\f p(\ft_1,\ldots,\ft_n)\equiv z_j$ for some
$j\neq i$, then $\fs_1\equiv \ft_1,\ldots,\fs_n\equiv \ft_n$, hence
$z_i\equiv z_j$ and $\tau(\fx_i)=\tau(\fx_j)$ by~\eqref{phipsi1}.\smallskip

The hypothesis now gives us that
$\A,\sigma\models\phi(z_1,\ldots,z_k) = \psi(z_{k+1},
\ldots,z_l)$, and we verify $\A,\tau\models\phi(\fx_1,\ldots,\fx_k) =
\psi(\fx_{k+1}, \ldots,\fx_l)$ by a direct computation:
\begin{multline*}
\tau(\A,\phi(\fx_1,\ldots,\fx_k)) = \phi^\A(\tau(\fx_1),\ldots,\tau(\fx_k))\\
= \phi^\A(\sigma(z_1),\ldots,\sigma(z_k))\quad(\text{by the construction})\\
= \psi^\A(\sigma(z_{k+1}),\ldots,\sigma(z_l))\quad(\text{by the hypothesis})\\
n=\psi^\A(\tau(\fx_{k+1}), \ldots, \tau(\fx_l))\quad(\text{by the construction})\\
= \tau(\A,\psi(\fx_{k+1},\ldots,\fx_l)).
\end{multline*}

The converse direction $(\Longleftarrow)$ of~\eqref{phipsi} is proved
similarly and we leave it for Problem~\ref{phipsiprob}.
\end{proof}

The lemma reduces the problem of the validity in $\A$ of any
identity in~\eqref{phipsieq} to the validity in $\A$ of a single
identity from a finite list and that would give us the decision method
we want by the same argument we used for total structures by appealing to
Lemma~\ref{totalirreducible}---except for the annoying subscript
${}_\vinj$, which we must deal with next.\smallskip

Consider the example in~\eqref{phipsiex}, for which the
construction in the Lemma verified that
\[
\A\models_\vinj \phi(\f p(\fs),\fs,\ft) = \psi(\f p(\ft),\fs,\ft) \iff
\A\models_\vinj \phi(\f q_1,\fs,\ft)=\psi(\f q_2,\fs,\ft),
\]
an equivalence which may fail without the subscript ${}_\vinj$ by Problem~\ref{problemex}. Now
\[
\A\models \phi(\f p(\fs),\fs,\ft) = \psi(\f p(\ft),\fs,\ft) \implies
\A\models \phi(\f p(\fs),\fs,\fs) = \psi(\f p(\fs),\fs,\fs);
\]
the construction in the Lemma associates the identity $\phi(\f q,\fs,\fs)=\psi(\f q,\fs,\fs)$
with $\phi(\f p(\fs),\fs,\fs)=\psi(\f p(\fs),\fs,\fs)$; and it is quite easy to check that
\begin{multline}
\label{phipsiex2}
\A\models\phi(\f p(\fs),\fs,\ft)=\psi(\f p(\ft),\fs,\ft) \\\iff
\A\models_\vinj\phi(\f q_1,\fs,\ft)=\psi(\f q_2,\fs,\ft) \conj \A\models_\vinj\phi(\f q,\fs,\fs)=\psi(\f q,\fs,\fs),
\end{multline}
\cf Problem~\ref{problemex2}. This reduces the validity of
$\phi(\f p(\fs),\fs,\ft)=\psi(\f p(\ft),\fs,\ft)$ in $\A$ to the injective validity
in $\A$ of two identities which are effectively constructed from
$\phi(\f p(\fs),\fs,\ft)=\psi(\f p(\ft),\fs,\ft)$; and the natural extension of this
reduction holds for arbitrary irreducible identities which involve
the primitives in $\Phi$:

\begin{lemma}
\label{phipsilemma2}

Let $\theta\equiv \phi(z_1.\ldots,z_k)=\psi(z_{k+1},\ldots,z_l)$
be any irreducible identity with $\phi,\psi\in\Phi$, let
$(\fs_1,\ldots,\fs_m)$ be an enumeration of the individual variables
which are placed in $\theta$, let $\EQP$ be the set of equivalence
relations on the set $\{\fs_1,\ldots,\fs_m\}$ and for any
$\sim\,\in\EQP$, let
\[
\widetilde \fs_i :\equiv \fs_j\text{ for the least $j$ such that $\fs_i\sim \fs_j$} \quad(i=1,\ldots,m).
\]
Then, for any infinite $\Phi$-structure $\A$,
\begin{equation}
\label{phipsilemma2eq}
\A\models\theta \iff \wwedge_{\sim\, \in\, \EQP} \A\models_\vinj
\theta\{\fs_i:\equiv \widetilde \fs_i\st i=1,\ldots,m\},
\end{equation}
where $\theta\{\fs_i:\equiv \widetilde \fs_i\st i=1,\ldots,m\}$ is the result of replacing
in $\theta$ each $\fs_i$ by its representative $\widetilde \fs_i$
in the equivalence relation $\sim$.
\end{lemma}

\begin{proof}

The direction $(\Longrightarrow)$ of~\eqref{phipsilemma2eq} is
trivial: because
\begin{multline*}
\A\models\theta \implies (\forall\sim\,\in\EQP)[\A\models\theta\{\fs_i:\equiv \widetilde \fs_i\st
i=1,\ldots,m\}]\\
\implies (\forall\sim\,\in\EQP)[\A\models_\vinj\theta\{\fs_i:\equiv \widetilde \fs_i\st i=1,\ldots,m\}].
\end{multline*}

For the converse, assume the right-hand-side
of~\eqref{phipsilemma2eq}, fix an assignment $\sigma$ and set
\[
\fs_i\sim \fs_j \iff \sigma(\fs_i)=\sigma(\fs_j) \qquad\text{(notice that $\sim$ depends on $\sigma$)};
\]
$\sigma$ is injective for the identity $\theta\{\fs_i:\equiv
\widetilde \fs_i\st i=1,\ldots,m\}$ whose placed individual
variables are exactly the representatives $\{\widetilde \fs_i \st
i=1,\ldots,m\}$ of the placed individual variables of $\theta$, and so the
hypothesis gives
\[
\A,\sigma\models\theta\{\fs_i:\equiv \widetilde \fs_i\st i=1,\ldots,m\}.
\]
To infer that $\A,\sigma\models\theta$, check
first that for every immediate term $z$,
\[
\sigma(z)= \sigma(z\{\fs_i:\equiv\widetilde \fs_i \st i=1,\ldots,m\})=\sigma(\widetilde z);
\]
this holds by the definition of $\sim$ if $z\equiv \fs_i$ for some
$i$ and then trivially in every other case. Finally, compute:
\begin{multline*}
\A,\sigma\models\theta\{\fs_i:\equiv \widetilde \fs_i\st i=1,\ldots,m\}\\
\implies \phi^\A(\sigma(\widetilde z_1),\ldots,\sigma(\widetilde z_k))
=\psi^\A(\sigma(\widetilde z_{k+1}),\ldots,\sigma(\widetilde z_l))\\
\implies\phi^\A(\sigma(z_1),\ldots,\sigma(z_k))=\psi^\A(\sigma(z_{k+1}),\ldots,\sigma(z_l))\\
\hspace*{2cm}\implies\sigma(\A,\phi(z_1,\ldots,z_k))=\sigma(\A,\psi(z_{k+1},\ldots,z_l))
\implies\A,\sigma\models\theta.\hspace*{.6cm}\qedsymbol
\end{multline*}
\noqed
\end{proof}
\vskip -10pt

A \textbf{bare} $\Phi$-\textbf{identity} is an identity of the form
\begin{equation}
\label{bare}
\phi(\fx_1,\ldots,\fx_k) = \psi(\fx_{k+1},\ldots,\fx_l),
\end{equation}
where each $\fx_i$ is an individual or nullary function variable of
sort $\ind$ chosen from the first $2l$ entries in a fixed list
\[
\f
v_0, \f p^{\ind,0}_0, \f v_1, \f p^{\ind,0}_1, \ldots
\]
of all such variables; and the \textbf{dictionary} of $\A$ is the set
\begin{equation}
\tag{Dictionary}
D(\A) = \{\theta\st \theta \text{ is a bare identity and }\A\models_\vinj\theta\}.
\end{equation}

The bare $\Phi$-identities are alphabetic variants of the
identities constructed in Lemma~\ref{phipsilemma} and there are
only finitely many of them, \cf Problem~\ref{barearity}.

\ynm{We understand~\eqref{bare} to include (by convention) identities of the form
\[
\phi(x_1,\ldots,x_k) = \psi, \quad \phi=\psi
\]
when the arity of one or both of $\phi,\psi$ is $0$. We will not
deal separately with these degenerate cases, since the arguments
we will give can be trivially adjusted to apply to them also.}

\begin{corollary}
\label{phipsicor}

With every irreducible identity of the form
\[
\theta \equiv \phi(z_1,\ldots,z_k)=\psi(z_{k+1},\ldots,z_l),
\]
we can associate a finite sequence $\theta_1,\ldots,\theta_n$ of bare
identities so that for every $\Phi$-structure $\A$,
\begin{equation}
\label{barerep}
\A\models\theta \iff \A\models_\vinj \theta_1 \conj\cdots\conj\A\models_\vinj\theta_n
\iff \theta_1, \ldots, \theta_n\in D(\A).
\end{equation}
\end{corollary}

\begin{proof}

Let $\sim_1,\ldots,\sim_n$ be an enumeration of the equivalence
relations on the set $\{\fs_1,\ldots,\fs_m\}$ of the placed individual
variables, let, for each $i$, $\theta'_i$ be the identity
associated with $\theta\{\fs_i:\equiv \widetilde \fs_i\st
i=1,\ldots,m\}$ when $\sim\,=\,\sim_i$ by the construction in
Lemma~\ref{phipsilemma}, and let $\theta_i$ be the ``alphabetically least'' variant
of $\theta'_i$ in which only the allowed variables occur.
\end{proof}

For
example, to be ridiculously formal, the bare identities associated
with Example~\eqref{phipsiex} ny~\eqref{phipsiex2} are
\[
\phi(\f p^{\ind,0}_0,\f v_0, \f v_1) = \psi(\f p^{\ind,0}_1,\f v_0,\f v_1) \text{ and }
\phi(\f p^{\ind,0}_0,\f v_0,\f v_0)=\psi(\f p^{\ind,0}_0,\f v_0,\f v_0).
\]

\begin{proofplain}[Proof of Theorem~\ref{localdec}]
is now immediate by appealing to Lemma~\ref{redto44}.
\end{proofplain}

\ynm{Following the plan for the proof on page~\pageref{proofplan}, we
reduce the problem of intensional equivalence between two extended
$\Phi$-programs $E(\vec{\f x})$ and $F(\vec{\f x})$ to deciding
the validity in $\A$ of irreducible term identities, and these fall
in one of ten cases i-j with $1\leq i\leq j\leq 4$
on the same page depending of the form of their two sides,
for example
\begin{equation}
\tag*{Case (3-4)} \tif w_1~\tthen w_2~\telse w_3 = \phi(z_1,\ldots,z_m).
\end{equation}

Case (1-1) is trivial, as no identity holds
between distinct terms which are immediate, or $\true$ or
$\false$, and
Lemma~\ref{2-2} and the first three parts of Lemma~\ref{irrthm}
decide the validity of the identity if $1\leq i\leq j\leq
3$ or $i=3, j=4$---for example, the identity in case (3-4) never holds by
Part (3) of Lemma~\ref{irrthm}.

Cases (1-4) and (2-4), $w = \phi(z_1,\ldots,z_k)$ where $w$ is $\true$,
$\false$ or an irreducible (perhaps immediate) term of the form
$p(w_1,\ldots,w_m)$ or $p$. Now part (4) of Lemma~\ref{irrthm}
reduces this identity to
\[
\phi(z_1,\ldots,z_k)=z_j
\]
for some $j$ (that we can compute), and the same holds in the
simpler case when $w$ is an individual or nullary function
variable, \cf Problem~\ref{14}. We assume that the identity  map
$\id(x)=x$ on $A$ is among the primitives---or we just add it,
which will just make the decidability result stronger---and then
this identity takes the form
\[
\phi(z_1,\ldots,z_k)=\id(z_i)
\]
and is covered by case (4-4) below. The same reduction
to (4-4) works for $\true$ and $\false$, we just assume that
these boolean constants have names in $\Phi$.\smallskip

Case (4-4) requires us to decide whether an irreducible
identity of the form
\[
\phi(z_1,\ldots,z_k)=\psi(z_{k+1},\ldots,z_l)
\]
holds in $\A$ and it is the only case which depends on the
primitives of $\A$. For this we will use the \textit{dictionary of
$\A$}, the set
\begin{equation}
\tag{Dictionary}
\label{dictionary}
D(\A) = \{\theta \st \theta \text{ is a bare identity and }\A\models_\vinj\theta\}.
\end{equation}
This is a finite set by Problem~\ref{barearity}, and by Corollary~\ref{phipsicor},
we can construct from $\phi(z_1,\ldots,z_k)$ and $\psi(z_{k+1},\ldots,z_l)$
bare identities $\theta_1,\ldots,\theta_n$ such that
\begin{multline*}
\A\models \phi(z_1,\ldots,z_k)=\psi(z_{k+1},\ldots,z_l)
\\
\hspace*{1cm}\iff \A\models_\vinj \theta_1\conj\dots\conj\A\models_\vinj\theta_n
\iff\theta_1,\ldots,\theta_n\in D(\A).\hspace*{.65cm} \qedsymbol
\end{multline*}
\noqed
\end{proof}}

\ysection{Global intensional equivalence}
We might guess that irreducible programs are globally
intensionally equivalent only if they are congruent, but this is
spoiled by the trivial,
\begin{equation}
\label{trivialpidentity}
\models E = \tif E~\tthen E~\telse E \quad(\text{with
$E$ explicit of boolean sort})
\end{equation}
already noticed in Problem~\ref{congintequalprob2}. So we need to
adjust:\smallskip

An irreducible program $E$ is \textbf{proper} if none of its parts
is an immediate term of the form $\f p$ or $\f p(\f u_1,\ldots,\f
u_n)$ and boolean sort.

\begin{theorem}
\label{globaldec}

\tu{1} Every extended irreducible $\Phi$-program is globally
intensionally equivalent to a proper one.

\tu{2} Two extended, proper, irreducible $\Phi$-programs are
globally intensionally equivalent if and only if they are
congruent.

\tu{3} The relation of global intensional equivalence
between extended $\Phi$-programs is decidable.

\end{theorem}

\begin{proof}

(1) Replace every part $E_i\equiv\underline{\f
p^{\bool,n}(\fx_1,\ldots,\fx_n)}$ of the program which is
immediate and of Boolean sort by $\underline{\tif E_i~\tthen
E_i~\telse E_i}$.\smallskip

(2) We need to show that if $E(\vec{\f x})$ and $F(\vec{\f x})$
are proper, extended irreducible $\Phi$-programs, then
\[
\Big(E(\vec{\f x})\inteq\A F(\vec{\f x}) \text{ for every infinite $\Phi$-structure $\A$}\Big)
\implies E\equiv_c F;
\]
and the idea is to use the hypothesis on a single, suitably
``free'' total structure $\A$ and then apply Lemma~\ref{totalirreducible}.\smallskip

We fix an infinite (countable) set $A$.\smallskip

\textit{Step 1}.  For each $\phi\in\Phi$ of sort $\ind$ and arity
$n\geq0$ (if there are any such), fix a set $R_\phi\subset
A$ and an injection $\phi^\A:A^n\inj R_\phi$, so that
\[
\phi\not\equiv\psi\implies R_\phi\cap R_\psi=\emptyset
\text{ and $\bigcup\{R_\phi \st \phi\in\Phi\}$ is co-infinite}.
\]
If $\arity(\phi)=0$, then $R_{\phi} = \{\oo\phi\}$ is a
singleton.\smallskip

It is easy---if a bit tedious---to check that for all immediate
terms $z_1,\ldots,z_k$, $w_1,\ldots,w_n$ and however we complete
the definition of $\A$, if $\phi$ and $\psi$ are of sort $\ind$,
then
\begin{multline}
\label{globaldec1}
\A\models\phi(z_1,\ldots,z_k) = \psi(w_1,\ldots,w_n)
\implies \phi(z_1,\ldots,z_k)\equiv\psi(w_1,\ldots,w_n),
\end{multline}
\cf Problem~\ref{step1}.\smallskip

\textit{Step 2}. For any two $\phi,\psi\in\Phi$ of sort $\bool$ and
arities $k>0, n>0$ and for every equivalence relation $\sim$ on the
set $\{1,\ldots,k,k+1,\ldots,k+n\}$, we choose a fresh tuple
$a_1,\ldots,a_k,a_{k+1},\ldots,a_{k+n}$ of elements in $A$ such
that
\[
a_i=a_j \iff i\sim j
\]
and set $\phi^\A(a_1,\ldots,a_k)=\true,\quad
\psi^\A(a_{k+1},\ldots,a_{k+n})=\false$. We do this by enumerating
all triples $(\phi,\psi,\sim)$ where $\phi,\psi$ are of boolean
sort and arities $k>0,n>0$ and $\sim$ is an equivalence relation
on $\{1,\ldots,k+n\}$, and then successively fixing a fresh tuple
$a_1,\ldots,a_{k+n}$ for each triple with the required
property---which we can do since, at any stage, we have used up
only finitely many members of $A$. This insures~\eqref{globaldec1}
for $\phi,\psi$ of sort $\bool$ and non-zero arities no
matter how we complete the definitions of $\phi^\A$.\smallskip

\textit{Step 3}. If $\gamma\in\Phi$  is of sort $\bool$ and arity $0$
(a boolean constant), we set $\gamma^\A\diverges$.\smallskip

At this point we have completed the definition of a structure $\A$
and the hypothesis gives us that
\[
\A\models E_i=F_i \quad (i=0, \ldots K).
\]

{\it Lemma 1. For $i=0,\ldots,K$, $E_i\equiv F_i$, unless both $E_i$
and $F_i$ are distinct boolean constants} in $\Phi$.

\begin{proofplain}[\it Proof]
\renewcommand\qedsymbol{\oldqedsymbol~(Lemma 1)}
is easy (if tedious) by appealing to Lemmas~\ref{purealgebraic} and \ref{2-3}
and we leave it for Problem~\ref{lemma1}.
\end{proofplain}

{\it Lemma 2. For each boolean constant $\gamma\in\Phi$},
\[
\Big|\{i\leq K \st E_i\equiv\gamma\}\Big| =
\Big|\{j\leq K \st F_j\equiv\gamma\}\Big|.
\]

\begin{proof}[\it Proof]
\renewcommand\qedsymbol{\oldqedsymbol~(Lemma 2)}

Fix $\gamma$ and define the structure $\B=\B^\gamma$ exactly as we defined $\A$ in Steps 1
and 2 above but replacing Step 3 by the following:\smallskip

\textit{Step $3^\gamma$}. Set $\gamma^\B=\true$ and for every other
boolean constant $\delta\in\Phi$, set $\delta^\B\diverges$.\smallskip

The hypothesis gives us a permutation $\pi:\{0,\ldots,K\}
\to\{0,\ldots, K\}$ such that $\pi(0)=0$ and
\[
\B\models E_i=F_{\pi(i)},\quad i\leq K.
\]
By a slight modification of the detailed case analysis in the
proof of Lemma~1 (in Cases \ef{2}{5},  \ef{3}{5} and \ef{4}{5}),
\begin{multline*}
\text{if $E_i$ is not a propositional constant in $\Phi$, or $\true$ or
$\false$,}\\
\text{then $F_{\pi(i)}$ is not a propositional constant in $\Phi$,
or $\true$ or $\false$,}
\end{multline*}
and directly from the definition of $\B$,
\begin{multline*}
\text{if $E_i$ is a propositional constant in $\Phi$ other than $\gamma$,}
\\
\text{then $F_{\pi(i)}$ is also a propositional constant in
$\Phi$ other than $\gamma$.}
\end{multline*}
So $\pi$ pairs the parts of $E$ which are $\true$ or $\gamma$ with
the parts of $F$ which are $\true$ or $\gamma$, hence
\[
\Big|\{i\leq K \st E_i\equiv\gamma \text{ or }E_i\equiv\true\}\Big| =
\Big|\{j\leq K \st F_j\equiv\gamma \text{ or }F_j\equiv\true\}\Big|;
\]
but $\Big|\{i\leq K \st E_i\equiv\true\}\Big| =
\Big|\{j\leq K \st F_j\equiv\true\}\Big|$ by Lemma~1, and so the
last displayed identity implies the claim in the Lemma.
\end{proof}

{\it Lemma 3}. $E_0\equiv F_0$.

\begin{proof}[\it Proof]
\renewcommand\qedsymbol{\oldqedsymbol~(Lemma 3)}

This is  true by Lemma~1, if $E_0$ is not a boolean constant in $\Phi$.

If $E_0\equiv\gamma$, then the construction in Lemma~2 gives
$F_0\equiv\true\vee F_0\equiv\gamma$; and the same construction
starting with $\gamma^\B=\false$ gives $F_0\equiv\false\vee
F_0\equiv\gamma$, so that
\[
(F_0\equiv\true\vee F_0\equiv\gamma)\conj
(F_0\equiv\false\vee F_0\equiv\gamma),
\]
which implies that $F_0\equiv \gamma$.
\end{proof}

To complete the proof that $E\equiv_cF$, we need to construct a
permutation $\rho:\{0,\ldots,K\}\bij\{0,\ldots,K\}$ such that
$\rho(0)=0$ and
\begin{equation}
\label{lastbit}
E_i\equiv F_{\rho(i)} \quad (i=1,\ldots,K).
\end{equation}

We start by setting  $\rho(i)=i$ if $E_i$ is not a boolean
constant in $\Phi$ or $i=0$, which insures~\eqref{lastbit} for
such $i$ by Lemmas~1 and~3. Next, for each boolean constant
$\gamma\in\Phi$, we appeal to Lemma~2 to extend $\rho$ in any way
so that
\[
\rho : \{i\leq K \st E_i\equiv\gamma\}\bij \{j\leq K\st F_j\equiv\gamma\}
\]
and establishes that $E\equiv_c F$.\smallskip

Part (2) follows almost immediately from (1) and we leave it for
Problem~\ref{cordecprob}.
\end{proof}

Together with the reduction calculus in Section~\ref{cforms},
Theorem~\ref{globaldec} suggests an obvious way to axiomatize the
relation of \textit{global intensional equivalence} between
programs, but we will not go into it here.

\problems

\begin{prob}
\label{1-123}
Give a decision procedure for $\A\models E=F$ in forms
\ef{1}{1}, \ef{1}{2}, \ef{1}{3} and arbitrary $\A$.
\end{prob}

\ynm{\begin{prob}[[Lemma~\ref{irrthm}, (3)]
\label{prob4d3}

Prove that if $z_1,\ldots,z_m$ and $w_1$, $w_2$, $w_3$ are immediate,
$\phi\in\Phi$ and $\A$ is any infinite $\Phi$-structure, then
\begin{equation}
\tag{$\ast$}
\label{prob4d3eq}
\A\not\models \phi(z_1,\ldots,z_m)=\tif w_1~\tthen w_2~\telse w_3.
\end{equation}

\sol{By Lemma~\ref{convdiv}, we can define some $\sigma$ on the
(individual) variables which occur in $\phi(z_1,\ldots,z_m)$ so
that it assigns convergent values $\oo z_1,\ldots,\oo z_m$ in $A$
to $z_1,\ldots,z_m$ and also $\oo w_2, \oo w_3\in A$ to $w_2,
w_3$, if they are of sort $\ind$---but it does not yet assign a
value to $w_1$, which is of boolean sort.

If $\phi^\A(\oo z_1,\ldots,\oo z_m)\converges$, extend $\sigma$
(by Lemma~\ref{convdiv} again) so that $\sigma(w_1)\diverges$; and
if $\phi^\A(\oo z_1,\ldots,\oo z_m)\diverges$, extend $\sigma$ so
that $\sigma(w_1)=\true$; and if $\sigma(w_2)$ has not been
defined yet, extend $\sigma$ again so that
$\sigma(w_2)\converges$. In any of these cases,
$\A,\sigma\not\models
\phi(z_1,\ldots,z_m)=\tif w_1~\tthen w_2~\telse w_3$.
}
\end{prob}}

\begin{prob}
\label{immcond1}

True or false: for any infinite $A$ and any immediate terms
$z_1,\ldots,z_m, w_1,w_2,w_3$ with $\sort(w_2)=\sort(w_3)=s\in\{\ind,\bool\}$,
we can define a partial function $\phi:A^m\pfto A_s$ such that
\[
(A,\phi)\models \phi(z_1,\ldots,z_m)=\tif~w_1~\tthen w_2~\telse w_3.
\]

\sol{False, this is just a restatement of (1) in Lemma~\ref{irrthm},
if you think about it.}
\end{prob}

\begin{prob}
\label{phipsi1prob}

Prove~\eqref{phipsi1}, that with the definitions in the proof of
Lemma~\ref{phipsilemma},
\begin{equation}
\tag{$\ast$}
\label{phipsi1probeq}
\fx_i\equiv \fx_j \iff z_i\equiv z_j \quad(1\leq i,j\leq l).
\end{equation}

\hint Use induction on $i$.

\end{prob}

\begin{prob}
\label{case11}

Prove that if $E$ and $F$ are distinct immediate terms or $\true$
or $\false$, then $\A\not\models E=F$, for any infinite $\A$.
\end{prob}

\begin{prob}

Suppose $\A$ is a total $\Phi$-structure and $\phi,\psi\in\Phi$.
Find a sequence of (not necessarily distinct) individual variables
$\fx_1,\ldots,\fx_6$ such that
\[
\A\models \phi(\f p(\fs),\fs,\ft) = \psi(\f p(\ft),\fs,\ft) \iff \A\models\phi(\fx_1,\fx_2,\fx_3)
=\psi(\fx_4,\fx_5,\fx_6).
\]
\sol{The method of proof of~\ref{totalirreducible} gives
\[
\A\models \phi(p(s),s,t) = \psi(p(t),s,t) \iff \A\models\phi(u,x,y)=\psi(v,x,y),
\]
which is easy to verify.}
\end{prob}

\begin{prob}
\label{problemex}
Give an example where
\[
\A\models\phi(\f p(\fs),\fs,\ft) = \psi(\f p(\ft),\fs,\ft) \text{ but }
\A\not\models\phi(\f q_1,\fs,\ft) = \psi(\f q_2,\fs,\ft).
\]
\hint Set $\phi^\A(x,s,t) = \psi^\A(y,s,t) = f(s,t)$, with a partial function $f(s,t)$
which converges only when $s=t$.
\sol{Set $\phi(\f r,s,t) = \psi(\f r,s,t) = \tif s=t~\tthen
\true~\telse\diverges$. Both sides of $\phi(\f p(s),s,t) =
\psi(\f p(t),s,t)$ diverge if $s\neq t$. If $s=t$, the identity
becomes $\phi(\f p(s),s,s) = \psi(\f p(s),s,s)$, and either both sides
diverge when $\f p(s)\diverges$ or both are $=\true$ when
$\f p(s)\converges$. The new identity becomes
$\phi\f (q_1,s,s)=\psi(\f q_2,s,s)$ and fails to hold when
$\f q_1\converges$ but $\f q_2\diverges$. }
\end{prob}

\begin{prob}
\label{phipsiprob}
Prove the direction ($\Longleftarrow$) of~\eqref{phipsi}, \ie
\begin{multline*}
\tag{$\ast$}
\label{phipsiprobeq}
\A\models_\vinj \phi(\fx_1,\ldots,\fx_k) = \psi(\fx_{k+1}, \ldots,\fx_l)\\
\implies \A\models_\vinj\phi(z_1,\ldots,z_k) = \psi(z_{k+1}, \ldots,z_l)
\end{multline*}
with the choice of variables $\fx_1,\ldots,\fx_l$ made in the proof of
Lemma~\ref{phipsilemma}.

\sol{Assume the hypothesis of~\eqref{phipsiprobeq},
suppose $\sigma$ is any injective (for the given equation)
assignment, set
\[
\tau(x_i) = \sigma(z_i) \quad (i=1,\ldots,l),
\]
and note directly from the definition that $\tau$ is an injective
assignment. Now compute,
\begin{multline*}
\sigma(\A,\phi(z_1,\ldots,z_k)) = \phi^\A(\sigma(z_1),\ldots,\sigma(z_k))
\\
=\phi^\A(\tau(x_1),\ldots,\tau(x_k))\quad(\text{by def})
\\
=\psi^\A(\tau(x_{k+1}),\ldots,\tau(x_l))\quad (\text{by hyp})
\\
=\psi^\A(\sigma(z_1),\ldots,\sigma(z_k))=\sigma(\A,\psi(z_{k+1},\ldots,z_l))
\end{multline*}
which is what we needed to show.
}
\end{prob}

\begin{prob}
\label{problemex2}
Prove~\eqref{phipsiex2}.
\sol{Take cases on whether $s\equiv t$ of $s\not\equiv t$.}
\end{prob}

\begin{prob}
\label{barearity}
Prove that the number of bare $\Phi$-identities as in \eqref{bare}
is bounded above by $2^{2\arity(\Phi)}|\Phi|^2$, where
$\arity(\Phi)$ is the largest arity of any $\phi\in\Phi$ and
$|\Phi|$ is its cardinality.

\sol{We are considering identities of the form
\[
\phi(x_1,\ldots,x_k)=\psi(x_{k+1},\ldots,x_l),
\]
where the $x_i$'s are individual of nullary function variable  of
sort $\ind$ from the first $l$ such in a fixed list, choosing each
$x_i$ to be the unique, first variable in that list that can be
used. There are 2 possibilities foe each $x_i$ for a total of
$2^l\leq 2^{2\arity(\Phi)}$ possibilities for each pair
$\phi,\psi$, and there are $|\Phi|^2$ such pairs.}
\end{prob}

\begin{prob}
\label{14}

Prove that if $\phi(z_1,\ldots,z_k)$ is irreducible, $\f w$ is an
individual variable and $\A\models \f w = \phi(z_1,\ldots,z_k)$ for
some structure $\A$, then every $z_i$ is an individual variable
and $\f w$ is one of them.

\sol{For every $\sigma$, $\sigma(w)\in A$, so
it is not possible to have $\sigma(w)=\phi^\A(\sigma(z_1),\ldots,\sigma(z_k))$ if
$\phi^\A(\sigma(z_1),\ldots,\sigma(z_k))\diverges$. So the hypothesis implies
that $\phi^\A(\sigma(z_1),\ldots,\sigma(z_k))\converges$ for every $\sigma$,
and then every $z_i$ must be a variable, otherwise we can get a
$\sigma$ such that $\sigma(z_i)\diverges$. Finally, if $w$ is
different from all the $z_i$'s, we can choose $\sigma$ that
violates the hypothesis by setting all the $\sigma(z_i)$ and then
choosing $\sigma(w)$ to be different from
$\sigma(\phi(z_1,\ldots,z_k))$.
}
\end{prob}

\begin{prob}
\label{step1}
Prove~\eqref{globaldec1}
\[
\A\models\phi(z_1,\ldots,z_k) = \psi(w_1,\ldots,w_n)
\implies \phi(z_1,\ldots,z_k)\equiv\psi(w_1,\ldots,w_n),
\]
in Step 1 of the  proof of Theorem~\ref{globaldec}.
\end{prob}

\begin{prob}
\label{lemma1}
Prove Lemma 1 in the proof of Theorem~\ref{globaldec}.

\end{prob}

\begin{prob}
\label{22}

Prove Case (2-2) of Theorem~\ref{globaldec}, that
for immediate terms $z_1,\ldots, z_l$,
\[
\models
p(z_1,\ldots,z_k) = q(z_{k+1},\ldots,z_l)
\implies p(z_1,\ldots,z_k) \equiv q(z_{k+1},\ldots,z_l).
\]

\end{prob}

\begin{prob}
\label{cordecprob}
Derive (2) of Theorem~\ref{globaldec} from (1).

\sol{Every irreducible program $E$ is globally intensionally equal to
the proper irreducible $E'$ constructed by replacing each part
$E_i$ of Boolean sort and of the form $\f p_j(\vec{\f u})$ by
$E_i':\equiv \tif \f p_j(\vec{\f u})~\tthen \f p_j(\vec{\f
u})~\telse \f p_j(\vec{\f u})$.}
\end{prob}

\ysection{Propositional programs}
An explicit term $E$ is \textit{propositional} if it has no
individual variables, no symbols from $\Phi$ and only boolean, nullary
function variables,
\begin{equation}
\tag{Prop terms}
\label{propterms}
P :\equiv \true \mid \false \mid \f p^{\bool, 0}_i \mid
\tif~P_1~\tthen P_2~\telse P_3;
\end{equation}
and a program is \textit{propositional} if all its parts are, \eg
\begin{multline*}
\text{Liar} :\equiv \f p\where\Lbrace \f p=\tif \f p~\tthen\false~\telse\true\Rbrace,
\\
\text{Truthteller}:\equiv \f p\where\Lbrace
\f p=\tif \f p ~\tthen \true~\telse\false\Rbrace.
\end{multline*}

These are significant for the project of using referential
intensions to model \textit{meanings} initiated
in~\cite{ynmfrege,ynmlcms} which is far from our topic, but they
also bear on the complexity problem for intensional
equivalence.\smallskip

For our purposes here, a \textit{finite graph with $n>0$ nodes} is
a binary relation $E$ on the set $\{0,\ldots,n-1\}$, the
\textit{edge relation} on the set of nodes of the graph
$(\{0,\ldots,n-1\}, E)$ in more standard terminology.\footnote{To
the best of my knowledge (and that of Wikipedia), it is still open
as I write this whether the graph isomorphism problem is co-NP.}

\begin{dprob}[\cite{ynmfrege}, for a related language]
\label{intequivlowerbound}
Prove that the problem of intensional
equivalence between propositional programs is at least as hard as
the problem of isomorphism between graphs. \hint You need to
associate with each graph $E$ of size $n$ an irreducible
propositional program $\prog E$ which codes $E$, so that
\[
E \text{ is isomorphic with }F \iff \prog E \inteq{}\prog F \iff
\prog E\equiv_c \prog F;
\]
and the trick is to use propositional variables $\f p_i$, one for
each $i<n$ and $\f p_{ij}$, one for each pair $(i,j)$ with $i,j<n$.

\end{dprob}

\def\test{}  
\def\dries{} 
\newcommand\fp{$\overline{\mathbb{F}}p$}
\bibliographystyle{ynmasl2}
\bibliography{ynmbasic}

\end{document}